\documentclass[a4paper,12pt]{amsart}

\usepackage{amssymb}
\usepackage{amsmath}
\usepackage{amsthm}
\usepackage{mathtools}
\usepackage{enumitem}
\usepackage{graphicx}
\usepackage{tikz}
\usepackage{mathrsfs}
\usepackage[mathscr]{euscript}
\usepackage{tkz-graph}

\usepackage{caption}
\usepackage[skip=1ex, belowskip=2ex]{subcaption}
\usepackage[export]{adjustbox}

\usetikzlibrary{arrows}
\usetikzlibrary{tikzmark}
\usetikzlibrary{cd}
\usetikzlibrary{calc}
\graphicspath{ {.} }

\usepackage[english]{babel}

\newcommand{\innerthmname}{}

\theoremstyle{definition}

\usetikzlibrary{positioning}
\newtheorem{theorem}[equation]{Theorem}
\newtheorem{lemma}[equation]{Lemma}

\newtheorem{corollary}[equation]{Corollary}

\theoremstyle{definition}
\newtheorem{definition}[equation]{Definition}

\theoremstyle{remark}
\newtheorem{example}[equation]{Example}
\newtheorem{remark}[equation]{Remark}

\numberwithin{equation}{section}
\usepackage{fancyhdr}
\usepackage{amsfonts}
\usepackage{hyperref}
\usepackage{graphicx, xcolor} 
\usepackage{amsthm}

\usepackage[
margin=1in,
marginpar=2cm,
includefoot,
footskip=30pt,
]{geometry}

\makeatletter
\@namedef{subjclassname@2020}{%
	\textup{2020} Mathematics Subject Classification}
\makeatother

\usepackage{scalerel,stackengine}
\stackMath
\newcommand\reallywidehat[1]{%
	\savestack{\tmpbox}{\stretchto{%
			\scaleto{%
				\scalerel*[\widthof{\ensuremath{#1}}]{\kern-.6pt\bigwedge\kern-.6pt}%
				{\rule[-\textheight/2]{1ex}{\textheight}}
			}{\textheight}%
		}{0.5ex}}%
	\stackon[1pt]{#1}{\tmpbox}%
}


\title[Partial Actions on Generalized Boolean Algebras]{Partial Actions on Generalized Boolean Algebras with Applications to Inverse Semigroups and Combinatorial $R$-Algebras}
\author[A. Zhang]{Allen Zhang}
\email{allenusca@gmail.com}

\begin{document}
	
	\begin{abstract}
		We define the notion of a partial action on a generalized Boolean algebra and associate to every such system and commutative unital ring $R$ an $R$-algebra. We prove that every strongly $E^{\ast}$-unitary inverse semigroup has an associated partial action on the generalized Boolean algebra of compact open sets of tight filters in the meet semilattice of idempotents. Using these correspondences, we associate to every strongly $E^{\ast}$-unitary inverse semigroup and commutative unital ring $R$ an $R$-algebra, and show that it is isomorphic to the Steinberg algebra of the tight groupoid. As an application, we show that there is a natural unitization operation on an inverse semigroup that corresponds to a unitization of the corresponding $R$-algebra. Finally, we show that Leavitt path algebras and labelled Leavitt path algebras can be realized as the $R$-algebra associated to a strongly $E^{\ast}$-unitary inverse semigroup.
	\end{abstract}
	\maketitle
	
	\section{Introduction}
	
	The interplay between inverse-semigroups and $C^{\ast}$-algebras is well-studied. Starting in \cite{Exel2007InverseSA}, Exel argued that there was arising a picture of forming a combinatorial $C^{\ast}$-algebra by passing through multiple stages. In particular, inverse semigroups were identified as a key object of interest. Inverse semigroups were first associated to a graph in Ash and Hall's paper \cite{Ash1975}, and many subsequent papers have studied the relation between graphs and their inverse semigroups more closely \cite{Ara2024InverseSO, bd76e05f99e74532899de92ac94a9f73, MEAKIN2021107729, Meakin21GraphInverse}. In \cite{Paterson1999}, Paterson developed theory associating a $C^{\ast}$-algebra to any inverse semigroup. 
	
	Seeking a more topological viewpoint, the tight filters and an associated topology on the set of tight filters was defined and studied by Exel and Lawson \cite{Exel2007TightRO, Lawson2011NonCommutativeSD}. In \cite{CARLSEN20162090}, Carlsen and Larsen showed that graph $C^{\ast}$-algebras could be realized from a topological partial action of a group on the tight filters of the associated inverse semigroup. 
	
	This theory was further pushed by characterizing the tight filters on the inverse semigroup associated to labelled spaces in \cite{Boava17Inverse} by Boava, de Castro, and Mortari. Using filter surgery with various gluing and cutting maps, a topological partial action that realized the $C^{\ast}$-algebra was defined using this characterization in \cite{deCastro2020} by de Castro and van Wyk. An analogous theory was developed by de Castro and Kang for the even more general generalized Boolean dynamical systems in \cite{Castro2022CalgebrasOG, DECASTRO2023126662}. 
	
	The resulting topological partial action was connected to the algebraic analog of the $C^{\ast}$-algebra by Boava, de Castro, Goncalves, and van Wyk in \cite{Boava2021LeavittPA}. We will primarily focus on the algebraic analog in this paper.
	
	A rich theory associating inverse semigroups, partial actions, and $C^{\ast}$-algebras was also developed in various papers \cite{Abadie2004OnPA, Kellendonk2004PartialAO, Milan2011ONIS}. In particular, Kellendonk and Lawson in \cite{Kellendonk2004PartialAO} defined a partial action on a strongly $E^{\ast}$-unitary inverse semigroup. Then, in Milan's thesis \cite{Milan2008Semigroup}, he used this partial action to realize the $C^{\ast}$-algebra of a strongly $E^{\ast}$-unitary inverse semigroup as a partial crossed product. Eventually, in \cite{Milan2011ONIS}, Milan and Steinberg further developed this theory by associating to a strongly $E^{\ast}$-unitary inverse semigroups a topological partial action, which realized the associated $C^{\ast}$-algebra.
	
	In this paper, we take a slightly different viewpoint of this process by inducing a partial action on a generalized Boolean algebra, rather than a topological space, from a strongly $E^{\ast}$-unitary inverse semigroup. We will show in a precise sense that this partial action on a generalized Boolean algebra is equivalent to the topological partial action on tight representations constructed in \cite{Milan2011ONIS}. However, there will be two notable differences in our presentation compared to Milan.
	
	The first will be the focus on the algebraic analog, rather than the $C^{\ast}$-algebra case. In most places, this will be a minor difference and one can compare the relevant sections in \cite{Boava2021LeavittPA} and \cite{deCastro2020} to see how the realization from the partial action differs in the $R$-algebra and $C^{\ast}$-algebra case respectively.
	
	The more important distinction will be to view the partial action as acting only on the generalized Boolean algebra of compact open sets. The main benefit that we derive from viewing the actions as on a generalized Boolean algebra, instead of a topological space,  is enlarging the set of natural morphisms between $R$-algebras that are induced from an underlying morphism of partial actions. In particular, we will show how natural subalgebras can arise nicely from a notion of a partial subaction on generalized Boolean algebras. As one application, we will showcase a natural unitization operation on a strongly $E^{\ast}$-unitary inverse semigroup that corresponds to a natural unitization operation on the associated $R$-algebra.
	
	Eventually, we will realize Leavitt path algebras and labelled Leavitt path algebras as the $R$-algebra associated to a strongly $E^{\ast}$-unitary inverse semigroup in the following diagram. 
	
		\tikzset{
		basic/.style={
			draw,
			rectangle,
			align=center,
			minimum width=8em,
			minimum height=3em,
			text centered
		},
		arrsty/.style={
			draw=black,
			-latex
		}
	}
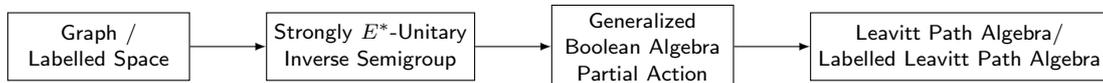
	\begin{center}
		\begin{tikzpicture}[every node/.style={font=\sffamily}]
			\tikzset{every node}=[font=\tiny]
			
			\node[basic] (a) {Graph / \\ Labelled Space};
			\node[basic, right=of a] (b) {Strongly $E^{\ast}$-Unitary \\ Inverse Semigroup};
			\node[basic, right=of b] (c) {Generalized \\ Boolean Algebra \\ Partial Action };
			\node[basic, right=of c] (d) {Leavitt Path Algebra/ \\Labelled Leavitt Path Algebra};
			
			\draw[arrsty]   (a)    --  (b);
			\draw[arrsty]   (b)    --  (c);
			\draw[arrsty]   (c)    --  (d);
			
		\end{tikzpicture}
	\end{center}
	
	This resolution will be an explicit calculation of the action on tight representations found in \cite{Milan2011ONIS}. A benefit of our more general construction from more recent realizations of these algebras as partial crossed products using topological partial actions is avoiding both the characterization of filters as in \cite{Boava17Inverse, DECASTRO2023126662} and subsequent definition of cutting and gluing maps to do filter surgery as in \cite{Castro2022CalgebrasOG, deCastro2020}. As a result, our actions can be more easily described.
	
	The structure of the paper is as follows. In Section~\ref{booleanaction}, we define a partial action on a generalized Boolean algebra and associate to each such partial action and unital commutative ring $R$ an $R$-algebra. In Section~\ref{meetlatticetight}, we develop some theory on tight filters that will be applied to construct a partial action on a generalized Boolean algebra from a strongly $E^{\ast}$-unitary inverse semigroup. In Section~\ref{inversegroupaction}, we associate to a strongly $E^{\ast}$-unitary inverse semigroup a partial action on a generalized Boolean algebra and show that it is isomorphic to the Steinberg algebra $A_R(\mathcal G_{tight}(S))$. As an application, we will define a natural unitization operation on a strongly $E^{\ast}$-unitary inverse semigroup that will correspond to a natural unitization operation on its associated $R$-algebra. Finally, in Section~\ref{isomorphism}, we will show that graph and labelled space $R$-algebras can be realized from a strongly $E^{\ast}$-unitary inverse semigroup using our general construction.

	\section{Partial Actions on Generalized Boolean Algebras and an Associated Skew Group Ring}\label{booleanaction}
	
	In this section, we define the notion of a partial action of a group $G$ on a generalized Boolean algebra and associate to such actions a partial skew group ring.
	
	\subsection{Generalized Boolean Algebras and Partial Actions}\label{booleanalgebras}
	\begin{definition}
		$\mathcal B$ is a \textit{generalized Boolean algebra} if it is a lattice that is distributed, relatively complemented, and has a bottom element $0$. We denote the meet operation as $\land: \mathcal B \times \mathcal B \rightarrow \mathcal B$, the join operation as $\vee: \mathcal B \times \mathcal B \rightarrow \mathcal B$, and the relative complement operation as $\setminus: \mathcal B \times \mathcal B \rightarrow \mathcal B$. If $\mathcal B$ has a top element, then it is \textit{unital} and called a \textit{Boolean algebra}. Morphisms of generalized Boolean algebras are maps of lattices that preserve the operations and bottom element.
	\end{definition}

	\begin{example}
		The finite subsets of $\mathbb N$ form a generalized Boolean algebra that is not a Boolean algebra.
		
	\end{example}
	\begin{example}
		For Hausdorff topological space $X$, the set of compact open sets denoted at $\mathcal O_c(X)$ forms a generalized Boolean algebra with joins as unions, meets as intersections, and relative complements as set subtraction.
	\end{example}
	
	\begin{example}
		Let $X_1$ and $X_2$ be Hausdorff spaces and let $f: X_1 \rightarrow X_2$ be a proper continuous map of topological spaces. There is a generalized Boolean algebra morphism $\mathcal O_c(X_2) \rightarrow \mathcal O_c(X_1)$ induced by the preimage map $U \mapsto f^{-1}(U)$.
	\end{example}
	
	\begin{definition}	
		A subset $\mathcal I \subseteq \mathcal B$ is called an \textit{ideal} if for any $a, b \in \mathcal I$, we have that $a \vee b \in \mathcal I$ and for any $a \in \mathcal I$ and $b \in \mathcal B$, we have that $a \wedge b \in \mathcal I$.
	\end{definition}

	Note that ideals of generalized Boolean algebras are generalized Boolean algebras in their own right as the relative complement is well-defined within the ideal.
	
	\begin{example} \label{ex:opencontains}
		For a Hausdorff topological space $X$ and the generalized Boolean algebra of compact open sets $\mathcal O_c(X)$, for any open set $U \subseteq X$ the set $\mathcal O_c(X)(U) = \{A \in \mathcal O_c(X) \colon A \subseteq U\}$ is an ideal of $\mathcal O_c(X)$.
	\end{example}
	
	\begin{remark}
		Generalized Boolean algebras have a natural order defined by \[x \leq y \Leftrightarrow x \wedge y = x\]
	\end{remark}

	\begin{definition}
		A subset $S \subseteq \mathcal B$ is called a \textit{cover} of $\mathcal B$ if for every $x \in \mathcal B$ there exists some $s_1, \ldots, s_n \in \mathcal S$ such that $x \leq \bigvee_{i=1}^n s_i$.
		
		A subset $S \subseteq \mathcal B$ \textit{generates} $\mathcal B$ if all elements of $\mathcal B$ can be written as the result of some finite sequence of operations involving elements from $S$. In other words, $S$ is contained in no proper generalized Boolean subalgebra of $\mathcal B$.
	\end{definition}
	
	\begin{example}
		Let $\mathbb Z$ be the set of integers and let $F(\mathbb Z)$ be the generalized Boolean algebra of finite subsets of $\mathbb Z$. Then $F(\mathbb Z)$ is covered by the subset $S = \{\{x, -x\}\}_{x \in \mathbb Z}$ because every finite subset of $X \in F(\mathbb Z)$ can be contained in the set $\bigcup_{x \in X} \{x, -x\}$. However, $S$ does not generate $F(\mathbb Z)$ because we cannot form singleton elements like $\{x\}$ for $x \neq 0$.
	\end{example}

	\begin{definition} \label{booldef}
		A partial action of a group $G$ on a generalized Boolean algebra $\mathcal B$ is a pair $\Phi = (\{\mathcal I_t\}_{t \in G}, \{\phi_t\}_{t \in G})$ consisting of ideals $\mathcal I_t \subseteq \mathcal B$ and generalized Boolean algebra isomorphisms $\phi_t: \mathcal I_{t^{-1}} \rightarrow \mathcal I_t$ such that
		
		\begin{enumerate}
			\item $\mathcal I_e = \mathcal B$ and $\phi_e$ is the identity on $\mathcal B$
			\item $\phi_s(\mathcal I_{s^{-1}} \cap \mathcal I_t) = \mathcal I_s \cap \mathcal I_{st}$ for all $s, t \in G$
			\item $\phi_s(\phi_t(x)) = \phi_{st}(x)$ for $x \in \mathcal I_{s^{-1}} \cap \mathcal I_{(st)^{-1}}$ for all $s, t \in G$
		\end{enumerate}
	\end{definition}
	
	We refer to $(\mathcal B, \Phi)$ as a partial action of $G$ on a generalized Boolean algebra if the underlying generalized Boolean algebra in $\Phi$ is $\mathcal B$ and the underlying group is $G$.
	
	\begin{example} \label{example:stonedualaction}
	We now briefly describe how our definition of a partial action on a generalized Boolean algebra both induces and can be induced from a partial action on a Hausdorff topological space. In some sense, this is an analog of Stone duality for partial actions on generalized Boolean algebras. For the definition of a topological partial action, see \cite[Definition~3.1]{deCastro2020}.

	 We first describe how a partial action on a generalized Boolean algebra can be induced from a partial action on a Hausdorff topological space.	
		
	Let $(\{V_t\}_{t \in G}, \{\phi_t\}_{t \in G})$ be a partial action on a Hausdorff topological space $X$. Because $V_t$ is an open subset of $X$, open compact sets of $V_t$, denoted as $\mathcal O_c(V_t)$, are exactly the open compact sets of $X$ contained within $V_t$, denoted as $\mathcal O_c(X)(V_t)$ (see Example~\ref{ex:opencontains}). Hence, we can view $\mathcal O_c(V_t) = \mathcal O_c(X)(V_t) \subseteq \mathcal O_c(X)$ as an ideal. 
	
	Furthermore, there is a natural isomorphism \[\tilde{\phi}_t: \mathcal O_c(V_{t^{-1}}) \rightarrow \mathcal O_c(V_t)\] between the generalized Boolean algebras of compact open sets as $\phi_t: V_{t^{-1}} \rightarrow V_t$ is an isomorphism of topological spaces. Note importantly that we do not ``pullback'' the map $\phi_t$, which would cause the morphism $\tilde{\phi_t}$ to go from $\mathcal O_c(V_t) \rightarrow \mathcal O_c(V_{t^{-1}})$. We then get a partial action on the generalized Boolean algebra $\mathcal O_c(X)$ with actions $\mathcal O_c(\Phi) = (\{\mathcal O_c(V_t)\}_{t \in G}, \{\tilde{\phi}_t\}_{t \in G})$.

	We next describe how a partial action on a generalized Boolean algebra induces a partial action on a topological space. Here, we make use of the Stone dual $X(\mathcal B)$ of a generalized Boolean algebra which is described in \cite[Section~2.4]{lawson2022non}.
		
	Let $(\{\mathcal I_t\}_{t \in G}, \{\phi_t\}_{t \in G})$ a partial action on a generalized Boolean algebra $\mathcal B$. The Stone dual $X(\mathcal B)$ is a Hausdorff topological space where the compact open sets correspond exactly to $\mathcal B$. Furthermore, the Stone duals of ideals $\mathcal I_t \subseteq \mathcal B$ correspond naturally to an open set $X(\mathcal I_t) \subseteq X(\mathcal B)$.
	
	The isomorphisms $\phi_t: \mathcal I_{t^{-1}} \rightarrow \mathcal I_t$ then induce an isomorphism on their respective Stone duals \[\tilde{\phi}_t: X(\mathcal I_{t^{-1}}) \rightarrow X(\mathcal I_t)\] We then get a partial action on the topological space $X(\mathcal B)$ with actions $X(\Phi) = (\{X(\mathcal I_t)\}_{t \in G}, \{\tilde{\phi}_t\}_{t \in G})$.

	\end{example}

	In our next definition, we present an analog of an orthogonal and semi-saturated partial action on a topological space. Assume that $G = \mathbb F[\mathcal A]$ is a free group with generators $\mathcal A$.
	
	\begin{definition} A partial action on a generalized Boolean algebra is said to be 
		\begin{enumerate} 
			\item \textit{orthogonal} if $\mathcal I_a \cap \mathcal I_b = \{0\}$ for $a \neq b \in \mathcal A$.
			\item \textit{semi-saturated} if $\phi_s \circ \phi_t = \phi_{st}$ for $s, t \in \mathbb F[\mathcal A]$ with $|s+t| = |s| + |t|$.
		\end{enumerate}
	\end{definition}
	
	\begin{remark}[{\cite[Proposition~4.1]{Exel2003}}] \label{equivsemi}
		The condition of being semi-saturated is equivalent to $\mathcal I_{st} \subseteq \mathcal I_{s}$ for any $s, t \in \mathbb F[\mathcal A]$ with $|s+t| = |s| + |t|$.
	\end{remark}

	\begin{definition} \label{def:booleanmorph}
		Let $(\mathcal B_1, \Phi_1)$ and $(\mathcal B_2, \Phi_2)$ be two partial actions of $G$ on generalized Boolean algebras. We define morphisms from $(\mathcal B_1, \Phi_1) \rightarrow (\mathcal B_2, \Phi_2)$ as maps $f: \mathcal B_1 \rightarrow \mathcal B_2$ that satisfy the following properties.
	\end{definition}
	\begin{enumerate}
		\item $f$ is a generalized Boolean algebra morphism $\mathcal B_1 \rightarrow \mathcal B_2$
		\item $f(\mathcal I_{1, t}) \subseteq \mathcal I_{2, t}$ where $\Phi_1 = (\{\mathcal I_{1, t}\}_{t \in G}, (\phi_{1, t})_{t \in G})$ and $\Phi_2$ is denoted similarly
		\item $f$ commutes with the actions $\Phi_1$ and $\Phi_2$ in the sense of the following diagram for all $t \in G$
			\begin{center}
		\begin{tikzcd}
			\mathcal I_{1, t^{-1}} \arrow[r, "f"] \arrow[d, "\phi_{1, t}"] & \mathcal I_{2, t^{-1}} \arrow[d, "\phi_{2, t}"] \\
			\mathcal I_{1, t} \arrow[r, "f"] & \mathcal I_{2, t} \\
		\end{tikzcd}
	\end{center}
	\end{enumerate}
	
	\subsection{An Associated Partial Skew Group Ring}
	Let $R$ be a unital commutative ring and $\mathcal B$. In this section, to a partial action of $G$ on $\mathcal B$, we associate an associative $G$-graded $R$-algebra in a similar fashion to the partial skew group ring associated to a topological partial action (see \cite[Section~4]{Boava2021LeavittPA}).
	
	\begin{definition}
		
		We define the $R$-algebra of compactly supported $\mathcal B$-functions as the set of set-theoretic functions \[\mathrm{Lc}(R,  \mathcal B) = \{f: R \setminus \{0\} \rightarrow \mathcal B \colon f(a) \wedge f(b) \neq \emptyset \Leftrightarrow a = b, f(a) = 0 \text { for all but finitely many } a \in R \setminus \{0\}\}\]
		
		For $f \in \mathrm{Lc}(R,  \mathcal B)$, define $\mathrm{dom}(f) = \bigvee_{r \in R \setminus \{0\}} f(r)$.
		
		We define an action of $r \in R$, addition, and multiplication as follows.
		
		\[(rf)(x) = \bigvee\limits_{y \in R \setminus \{0\} \colon ry = x, f(y) \neq 0} f(y) \neq 0\]\[(f+g)(x) = (f(x) \setminus \mathrm{dom}(g)) \vee (g(x) \setminus \mathrm{dom}(f))\bigvee\limits_{r \in R \setminus \{0, x\}} (f(r) \vee g(x-r))\] \[(fg)(x) = \bigvee\limits_{r_1, r_2 \in R \setminus \{0\} \colon r_1r_2 = x} (f(r_1)\wedge f(r_2))\]
	\end{definition}
	
	In the below lemma, we give a more natural description of the $R$-algebra $\mathrm{Lc}(R,  \mathcal B)$.
	
	\begin{lemma} \label{stonedual}
		Let $X(\mathcal B)$ be the Stone dual of $\mathcal B$. Then $\mathrm{Lc}(R,  \mathcal B)$ is isomorphic as an $R$-algebra to the set of locally constant functions with compact support from $X(\mathcal B) \rightarrow R$ where $R$ has the discrete topology.
	\end{lemma}
	
	\begin{proof}
		Let $\mathrm{Lc}(X(\mathcal B), R)$ denote the locally constant functions with compact support on $X(\mathcal B)$. Our goal is to prove that $\mathrm{Lc}(X(\mathcal B), R) \cong \mathrm{Lc}(R, \mathcal B)$.
		
		By Stone duality, $X(\mathcal B)$ is Hausdorff and there is a generalized Boolean algebra isomorphism \[\mathcal O_c(X(\mathcal B)) \xrightarrow{\sim} \mathcal B\]
		
		We first draw a bijection between the sets $\mathrm{Lc}(X(\mathcal B), R)$ and $\mathrm{Lc}(R, \mathcal B)$. 
		
		Let $f \in \mathrm{Lc}(X(\mathcal B), R)$. We associate to $f$ the function from $R \setminus \{0\} \rightarrow \mathcal B$ defined by $r \mapsto f^{-1}(r)$ for $r \in R \setminus \{0\}$. Note that $\{r\} \subseteq R$ is closed by definition, hence $f^{-1}(r)$ must be closed. Because $f$ has compact support, $f^{-1}(r)$ is a closed set in a compact space, so, because $X(\mathcal B)$ is Hausdorff, $f^{-1}(r)$ is compact. Because $f$ is locally constant, $f^{-1}(r)$ is also open. Thus, we conclude that $f^{-1}(r) \in \mathcal O_c(X(\mathcal B))$. 
		
		Obviously, all $\{f^{-1}(r)\}_{r \in R \setminus \{0\}}$ are pairwise disjoint. To prove that only a finite number of $f^{-1}(r)$ can be not the empty set, assume for the sake of contradiction that there are an infinite number $S = \{r_1, r_2, \ldots\}$ such that $f^{-1}(r_i) \neq \emptyset$ and $r_i \neq 0$. Then we note that $\{r_1, r_2, \ldots\}$ is closed as well, so $\bigcup_{r \in S}  f^{-1}(r)$ is closed. Because $f$ has compact support, $\bigcup_{r \in S} f^{-1}(r)$ is compact. However, each $f^{-1}(r)$ is an open set as well, hence we can cover our compact set as the disjoint infinite union of non-empty open sets, which is a contradiction. Hence, our map is well-defined.
		
		To construct the inverse, for some $f \in \mathrm{Lc}(R, \mathcal B)$, we can associate to $f$ the function from $X(\mathcal B) \rightarrow R$ by taking the sum $\sum_{r \in R \setminus \{0\}} r 1_{f(r)}$ on $X(\mathcal B)$, interpreting $f(r) \in \mathcal O_c(X(\mathcal B))$. It's obvious this function is in $\mathrm{Lc}(X(\mathcal B), R)$ because $r1_{f(r)}$ is in $\mathrm{Lc}(X(\mathcal B), R)$.
		
		The proof that these maps are inverses of each other is easy. Furthermore, it's easy to see that operations in $\mathrm{Lc}(R, \mathcal B)$ are equivalent to the corresponding operations in the function space $\mathrm{Lc}(X(\mathcal B), R)$, so we are done.
		
	\end{proof}
	
	\begin{definition}
		For $U \in \mathcal B$, define $1_U \in \mathrm{Lc}(R,  \mathcal B)$ to be the function $R \setminus \{0\} \rightarrow \mathcal B$ that takes $1 \mapsto U$ and all other elements to $0$.
	\end{definition}
	
	\begin{lemma}
		The set $\{1_U\}_{U \in \mathcal B}$ $R$-spans the $R$-algebra $\mathrm{Lc}(R,  \mathcal B)$.
	\end{lemma}
	\begin{proof}
		For $f \in \mathrm{Lc}(R, \mathcal B)$, we have that $f = \sum_{r \in R \setminus \{0\}, f(r) \neq 0} r 1_{f(r)}$ where the sum is finite by definition of $\mathrm{Lc}(R, \mathcal B)$.
	\end{proof}

	\begin{lemma}\label{idealfunction}
		For any ideal $\mathcal I \subseteq \mathcal B$, there is a natural injection of $R$-algebras $\mathrm{Lc}(R, \mathcal I) \hookrightarrow \mathrm{Lc}(R,  \mathcal B)$ composing $f: R \setminus \{0\} \rightarrow \mathcal I $ with the injection $\mathcal I \hookrightarrow \mathcal B$. Furthermore, $\mathrm{Lc}(R, \mathcal I) \subseteq \mathrm{Lc}(R,  \mathcal B)$ is an ideal as rings.
	\end{lemma}
	
	\begin{proof}
		It's obvious that the mapping is injective and is a morphism of $R$-algebras, so it just remains to prove that $\mathrm{Lc}(R, \mathcal I) \subseteq \mathrm{Lc}(R,  \mathcal B)$ is an ideal.
		
		From the previous lemma, $\{1_U\}_{U \in \mathcal I}$ $R$-spans $\mathrm{Lc}(R, \mathcal I)$ and $\{1_V\}_{V \in \mathcal B}$ $R$-spans $\mathrm{Lc}(R, \mathcal B)$. Thus, we just need to check that for any $U \in \mathcal I$ and any $V \in \mathcal B$ we have that $1_U \cdot 1_V, 1_V \cdot 1_U \in \mathrm{Lc}(R, \mathcal I)$. It's easy to check that $1_U \cdot 1_V = 1_{U \wedge V} = 1_V \cdot 1_U$.  If $\mathcal I \subseteq \mathcal B$ is an ideal and $U \in \mathcal I$, then $U \wedge V = V \wedge U \in \mathcal I$ so $1_{U \wedge V} \in \mathrm{Lc}(R, \mathcal I)$.
	\end{proof}
	
	\begin{remark}
		Continuing the parallel with Stone duality from Lemma~\ref{stonedual}, $\mathrm{Lc}(R, \mathcal I)$ can be thought of the functions of $\mathrm{Lc}(X(B), R)$ that are supported within the open set $X(\mathcal I) \subseteq X(\mathcal B)$.
	\end{remark}

	\begin{definition}
		Let $\Phi = (\{\mathcal I_t\}_{t \in G}, \{\phi_t\}_{t \in G})$ be a partial action on $\mathcal B$. We define a dual partial action of $G$ on $\mathrm{Lc}(R,  \mathcal B)$ with (ring) ideals $D_t = \mathrm{Lc}(R, \mathcal I_t)$ (by Lemma~\ref{idealfunction}). We define an $R$-algebra isomorphism $\tilde{\phi_t}: D_{t^{-1}} \rightarrow D_t$ by \[\tilde{\phi_t}(f) = \phi_t \circ f\]
		
		This induces the partial skew group ring \[\mathrm{Lc}(R,  \mathcal B) \rtimes_{\Phi} G = \left\{\sum_{t \in G} f_t \delta_t \colon f_t \in D_t,  f_t = 0 \text{ for all but a finite number of } t \in G\right\}\] where multiplication is defined by $(a \delta_s)(b \delta_{t}) = \tilde{\phi}_s(\tilde{\phi}_{s^{-1}}(a)b)\delta_{st}$.
	\end{definition}
	
	\begin{remark}
		$\mathrm{Lc}(R,  \mathcal B) \rtimes_{\Phi} G$ has a natural $G$-grading given by $a \delta_g \mapsto g$.
	\end{remark}

	\begin{remark}  \label{fulldual}
		For a Hausdorff topological space $X$, we have that the $R$-algebra of locally constant functions with compact support $\mathrm{Lc}(X, R)$ is isomorphic to our $R$-algebra $\mathrm{Lc}(R, \mathcal O_c(X))$. Furthermore, for any topological partial action, we can construct a similar dual action on $\mathrm{Lc}(X, R)$ (see \cite[Section~4]{Boava2021LeavittPA} for an example).
		
		If we have a topological partial action $(X, \Phi)$ and get a partial action $(\mathcal O_c(X), \mathcal O_c(\Phi))$ on a generalized Boolean algebra using Example~\ref{example:stonedualaction}, it's not difficult to show that their dual actions on $\mathrm{Lc}(X, R) \cong \mathrm{Lc}(R, \mathcal O_c(X))$ are equivariant. Hence, the resulting $R$-algebras are isomorphic \[\mathrm{Lc}(X, R) \rtimes G \cong \mathrm{Lc}(R, \mathcal O_c(X)) \rtimes G\]
		
		With similar reasoning, if we start with a partial action on a generalized Boolean algebra $(\mathcal B, \Phi)$ and construct the action on the Stone dual $(X(\mathcal B), X(\Phi))$, we have that \[\mathrm{Lc}(X(\mathcal B), R) \rtimes G \cong \mathrm{Lc}(R, \mathcal B) \rtimes G\]
		
		Hence, if the $R$-algebra is the only object of interest then there is no reason to discuss partial actions on generalized Boolean algebras rather than partial actions on Hausdorff topological spaces. However, we stay in the language of generalized Boolean algebras as it is often the case that maps between $R$-algebras that we are interested in can be best understood as maps of generalized Boolean algebras, rather than topological spaces.
	\end{remark}
	
	From now on, let $e \in G$ refer to the unit of $G$. As shorthand, for $g \in G$ and $U \in \mathcal I_g$, we refer to $1_U \delta_g \in \mathrm{Lc}(R, \mathcal B) \rtimes_{\Phi} G$ as $U \delta_g$.
	
	\begin{lemma} \label{computationskew}
		Below is a collection of some computations which will be used throughout.
		
		\begin{enumerate}
			\item $\mathrm{Lc}(R,  \mathcal B) \rtimes_{\Phi} G$ is generated by the $R$-span of elements of the form $\{U \delta_g\}_{g \in G, U \in \mathcal I_g}$.
			\item $1_U 1_V = 1_{U \wedge V}$ for $U, V \in \mathcal B$
			\item $\tilde{\phi}_{g}(1_U) = 1_{\phi_{g}(U)}$ for $g \in G$ and $U \in \mathcal I_{g^{-1}}$
			\item $(U \delta_{g})(V \delta_{g'}) = \phi_g(V \wedge \phi_{g^{-1}}(U)) \delta_{gg'}$ for $U \in \mathcal I_g$ and $V \in \mathcal I_{g'}$.
			\item $(U \delta_e)(V \delta_g) = (U \wedge V) \delta_g$ for $U \in \mathcal B$ and $V \in \mathcal I_g$.
			\item $(V \delta_g) (U \delta_e) = \phi_g(U \wedge\phi_{g^{-1}}(V)) \delta_g$ for $U \in \mathcal B$ and $V \in \mathcal I_g$.
			\item $(U' \delta_e)(V \delta_g)(U \delta_e) =  (U' \wedge \phi_g(U \wedge \phi_{g^{-1}}(V))) \delta_g$ for $U, U' \in \mathcal B$ and $V \in \mathcal I_g$.
			\item $(U \delta_g)(V \delta_e)(\phi_{g^{-1}}(U)\delta_{g^{-1}}) = \phi_g(V \wedge \phi_{g^{-1}}(U))\delta_e$ for $g \in G$, $V \in \mathcal B$, and $U \in \mathcal I_g$.
			\item $(U \delta_g)(\phi_{g^{-1}}(U)\delta_{g^{-1}}) = U\delta_e$ for $g \in G$ and $U \in \mathcal I_g$.
			\item Let $U \in \mathcal I_g$ and $\{U_1, \ldots, U_n\} \in \mathcal I_g$ such that $\bigvee_{i=1}^n U_i = U$. Then \[U\delta_g = \sum_{\emptyset \neq J \subseteq [1, \ldots, n]} (-1)^{|J|-1}\left(\bigwedge_{j \in J} U_j\right)\delta_g\]
		\end{enumerate}
	\end{lemma}
	
	\begin{proof} We prove each statement separately
		\begin{enumerate}
			\item Clear because $\mathrm{Lc}(R,  \mathcal I_g)$ is spanned by $\{1_U\}_{U \in \mathcal I_g}$
			\item Clear
			\item Clear from definition of $\tilde{\phi}_{g^{-1}}$
			\item From definition of multiplication and (3), $(U \delta_g) (V \delta_{g'}) = \tilde{\phi}_g(\tilde{\phi}_{g^{-1}}(1_U)1_V) \delta_{gg'} = \tilde{\phi}_g(\phi_{g^{-1}}(U) \wedge V) \delta_{gg'} = \phi_g(U \wedge \phi_{g^{-1}}(V)) \delta_{gg'}$
			\item Follows from (4)
			\item Follows from (4)
			\item  Follows by application of (6) then (5)
			\item If we apply (5), we get $(U \delta_g)(V \wedge \phi_{g^{-1}}(U) \delta_{g^{-1}})$. Now applying (4), this is $\phi_g((V \wedge \phi_{g^{-1}}(U)) \wedge \phi_{g^{-1}}(U)) \delta_{e} = \phi_g(V \wedge \phi_{g^{-1}}(U))\delta_e$
			\item Same proof as (8)
			\item Standard inclusion exclusion
		\end{enumerate}
	\end{proof}

	We now show that the partial skew group ring generated by a partial action on a generalized Boolean algebra has a nice set of local units.
	
	\begin{definition} A ring $S$ (not necessarily unital) is a \textit{ring with local units} if there exists a subset $E \subseteq S$ such that all elements of $E$ are idempotent, commute with each other, and for all $s \in S$ there exists some $e \in E$ where $es = s = se$. 
	\end{definition}
	
	\begin{definition}
		Let $e_1, e_2$ be commuting idempotents in a ring with local units. We define the \textit{meet} of idempotents as $e_1 \wedge e_2 =  e_1e_2$ and the \textit{join} of idempotents as $e_1 \vee e_2 = e_1 + e_2 - e_1\wedge e_2$. Note that $e_1 \wedge e_2$ and $e_1 \vee e_2$ are both idempotent.
	\end{definition}
	
	\begin{remark}\label{localadd}
		If $e_1, e_2$ are commuting idempotents and $x_1, x_2 \in S$ satisfy that $e_1x_1 = x_1$ and $e_2x_2 = x_2$, then $(e_1 \vee e_2)(x_1+x_2) = x_1+x_2$.
	\end{remark}
	
	\begin{theorem}		 \label{localunits}
		The set $\{U \delta_{e}\}_{U \in \mathcal B}$ forms a set of local units of $\mathrm{Lc}(R,  \mathcal B) \rtimes_{\Phi} G$ closed under the join and meet operations.
	\end{theorem}
	\begin{proof}
		The elements in $\{U \delta_{e}\}$ are idempotent, commute with each other, and are closed under meets using (5) in Lemma~\ref{computationskew}. Applying (10) in Lemma~\ref{computationskew}, we find that $(U \delta_e) \vee (V \delta_e) = U \delta _e + V \delta_e - (U \wedge V) \delta_e = (U \vee V)\delta_e$, which shows that the set is closed under joins.
		
		$\mathrm{Lc}(R,  \mathcal B) \rtimes_{\Phi} G$ is generated by the $R$-span of elements of the form $\{V \delta_g\}$ for arbitrary $V \in \mathcal I_g$ and $g \in G$. Hence, it suffices to prove that each of these elements have a unit inside $\{U\delta_e\}_{U \in \mathcal B}$ because we can otherwise take joins and use Remark~\ref{localadd}. For an arbitrary element $V \delta_g$, consider the element $U = V \vee \phi_{g^{-1}}(V)$. Using (7) in Lemma~\ref{computationskew}, we compute that $(U \delta_{e})(V \delta_g)(U \delta_e) = (V \delta_g)$, so we are done.
	\end{proof}
	
	\begin{corollary} \label{unitalcor}
		$\mathrm{Lc}(R,  \mathcal B) \rtimes_{\Phi} G$ is unital if and only if $\mathcal B$ is unital.
	\end{corollary}

	\begin{proof}
		It's not hard to prove that if $1 \in \mathcal B$ is a unit, then $1 \delta_e$ is a unit of $\mathrm{Lc}(R,  \mathcal B) \rtimes_{\Phi} G$ because it is maximal in a set of local units.
		
		In the other direction, let $x \in \mathrm{Lc}(R,  \mathcal B) \rtimes_{\Phi} G$ be a unit. Because $\{U \delta_e\}_{U \in \mathcal B}$ is a set of local units, there must be some $U \in \mathcal B$ such that $(U \delta_e)x = x$. Because $x$ is a unit however, we find that $(U \delta_e)x = U \delta_e$ so $x = U \delta_e$ for some $U \in \mathcal B$. 
		
		We now consider $V \delta_e$ for any $V \in \mathcal B$. From Lemma~\ref{computationskew} (5) and the fact that $(U \delta_e)$ is a unit, we find that $V \delta_e = (U \delta_e)(V \delta_e) = (U \wedge V) \delta_e$. This implies that $U \wedge V = V$ for arbitrary $V \in \mathcal B$. Hence, $U \in \mathcal B$ is a unit, and we are done.
	\end{proof}

	We now show that if $G = \mathbb F[\mathcal A]$ is a free group for some generating set $\mathcal A$ and the underlying partial action is semi-saturated, then the $R$-algebra has a simple set of generators.
	
	\begin{theorem}\label{saturatedgeneratingset}
		Let $(\mathcal B, \Phi)$ be a semi-saturated partial action of $G = \mathbb F[\mathcal A]$ on a generalized Boolean algebra. Let $C \subseteq \mathcal B$ be a set that generates $\mathcal B$ as a generalized Boolean algebra. For all $a \in \mathcal A$, let $C_a$ and $C_{a^{-1}}$ be a cover of $\mathcal I_a$ and $\mathcal I_{a^{-1}}$ respectively.
		
		Then, $\mathrm{Lc}(R,  \mathcal B) \rtimes_{\Phi} G$ is generated as an $R$-algebra by elements of the form \[\{U \delta_e\}_{U \in C}\cup\{V \delta_x\}_{V \in C_a, a \in A}\cup\{V \delta_{a^{-1}}\}_{V \in C_{a^{-1}}, a \in A}\]
	\end{theorem}
	
	\begin{proof}
		Let $A_R$ be the $R$-algebra generated by the given elements. By Theorem~\ref{computationskew} (1), it suffices to show that for all $s \in \mathbb F[\mathcal A]$ and $U \in \mathcal I_s$, we have that $U \delta_{s} \in A_R$.
		
		The condition that $C \subseteq \mathcal B$ generates $\mathcal B$ implies that $\{U\delta_e\}_{U \in \mathcal B} \subseteq A_R$ because for some pair $U \delta_e$ and $V \delta_e$, we can implement the generalized Boolean algebra operations using ring operations as follows: 
		
		\[(U \wedge V) \delta_e = (U \delta_e)(V \delta_e)\]
		\[(U \vee V)\delta_e = U \delta_e + V \delta_e - (U\wedge V)\delta_e\]
		\[(U \setminus V) \delta_e = U \delta_e - (U \wedge V) \delta_e\]
		
		Now, because $U \delta_e \in A_R$ for any $U \in \mathcal B$ and $(U \delta_e)(V \delta_s) = (U \vee V)\delta_s$, if $U \leq V \in \mathcal I_s$ where $V \delta_s \in A_R$, then $U \delta_s \in A_R$. 
		
		Because $C_a$ is a cover, any $V \in \mathcal I_a$ satisfies $V \leq \bigvee_{i=1}^n V_i$ for some $V_i \in C_a$. By (10) in Lemma~\ref{computationskew}, $\bigvee_{i=1}^n V_i \delta_a$ can be written as the sum of its inclusion exclusion components. However, note that each term in the inclusion exclusion formula is less than $V_i$ for some $i$. Hence, because $V_i \delta_a \in A_R$, all terms in our inclusion exclusion sum are in $A_R$ so $\bigvee_{i=1}^n V_i \delta_a \in A_R$. Thus, we find that $V \delta_a \in A_R$. A similar proof holds for $C_{a^{-1}}$, so we conclude that $\{V \delta_{a}\}_{V \in \mathcal I_a}$ and $\{V \delta_{a^{-1}}\}_{V \in \mathcal I_{a^{-1}}}$ are in $A_R$ for all $a \in \mathcal A$.
		
		Hence, we have proven that \[\{U \delta_e\}_{U \in \mathcal B}\cup\{V \delta_a\}_{V \in \mathcal I_a, a \in \mathcal A}\cup\{V \delta_{a^{-1}}\}_{V \in \mathcal I_{a^{-1}}, a \in \mathcal A} \subseteq A_R\]
		
		Using the semi-saturated assumption, we now prove inductively that for all $s \in \mathbb F[\mathcal A]$ and $U \in \mathcal I_s$ we have that $U \delta_s \in A_R$. Our previous statements prove this for words of length $0$ and $1$. Now let $s$ be a reduced word such that $|s| > 1$. We can split up our reduced word $s$ into a reduced representation $s_1, s_2$ where $s = s_1s_2$, $|s_1|, |s_2| > 0$, and $|s|= |s_1| + |s_2|$. 
		
		By the inductive hypothesis, we have that $\{V \delta_{s_1}\}_{V \in \mathcal I_{s_1}}, \{U \delta_{s_2}\}_{U \in \mathcal I_{s_2}} \subseteq A_R$. It suffices to prove that the image of \[(V \delta_{s_1}) (U \delta_{s_2})\] over all $V \in \mathcal I_{s_1}, U\in\mathcal I_{s_2}$ is $\mathcal I_{s_1s_2}$. We have that \[(V \delta_{s_1}) (U \delta_{s_2}) = \phi_{s_1}(U \cap \phi_{s_1}^{-1}(V)) \delta_{s}\] Over all possible $V \in \mathcal I_{s_1}$, we have that $\phi_{s_1}^{-1}(V)$ ranges over all possible $V' \in \mathcal I_{s_1^{-1}}$ so we can instead consider $\phi_{s_1}(U \cap V') \delta_{s}$ for arbitrary $V' \in \mathcal I_{s_1^{-1}}$.
				
		Over all  $U \in \mathcal I_{s_2}$ and $V' \in \mathcal I_{s_1^{-1}}$, we have that $\phi_{s_1}(U \cap V')$ ranges over all $\mathcal I_{s_1} \cap \mathcal I_{s_1s_2}$ by (2) in Definition~\ref{booldef}. Thus, because $s_1s_2$ is a reduced representation and our action is semi-saturated, we have that $\mathcal I_{s_1s_2} \subseteq \mathcal I_{s_1}$ so our image is $\mathcal I_{s_1s_2}$. Hence, $\phi_{s_1}(U \cap V')$ ranges over all sets in $\mathcal I_{s_1s_2}$, so we are done.
	\end{proof}

	We now describe how morphisms of partial actions on generalized Boolean algebras become $G$-graded morphisms of their associated $R$-algebras.
	
	\begin{theorem}
		Let $f:(\mathcal B_1, \Phi_1) \rightarrow (\mathcal B_2, \Phi_2)$ be a morphism of partial actions on generalized Boolean algebras. Then, there is an induced $G$-graded morphism of $R$-algebras \[\tilde{f}: \mathrm{Lc}(R, \mathcal B_1) \rtimes_{\Phi_1} G \rightarrow \mathrm{Lc}(R, \mathcal B_2) \rtimes_{\Phi_2} G\] taking \[a \delta_t \mapsto (f \circ a) \delta_t\]
	\end{theorem}
	\begin{proof}
		Because $f(\mathcal I_{1, t}) \subseteq \mathcal I_{2, t}$, the map is well-defined. 
		
		Note that for any fixed $t\in G$, $a, b \in \mathrm{Lc}(\mathcal I_t, R)$ and $r \in R$, we have that $f(ab) = f(a)f(b)$, $f(a+b) = f(a) + f(b)$, and $f(ra) = rf(a)$ because the ring operations in $\mathrm{Lc}(R, \mathcal B)$ are implemented using generalized Boolean algebra operations, which are preserved by $f$. It's not hard to show that this implies that $\tilde{f}$ preserves scalar multiplication and addition.
		
		To show that the  $\tilde{f}$ preserves multiplication, let $a \delta_s, b \delta_t \in \mathrm{Lc}(R, \mathcal B_1) \rtimes_{\Phi_1} G$. We have that \[(a \delta_s)(b \delta_t) = \phi_{1, s}(\phi_{1, s^{-1}}(a)b) \delta_{st} \]
		
		Thus, \[\tilde{f}((a\delta_s)(b\delta_t)) = f(\phi_{1, s}(\phi_{1, s^{-1}}(a)b)) \delta_{st} = \phi_{2, s}(f(\phi_{1, s^{-1}}(a)b)) \delta_{st} = \]
		\[\phi_{2, s}(f(\phi_{1, s^{-1}}(a))f(b)) \delta_{st} = \phi_{2, s}(\phi_{2,s^{-1}}(f(a))f(b)) \delta_{st} \] where we use the commutativity of the diagram in (3) of Definition~\ref{def:booleanmorph}.
		
		Similarly, we calculate
		
		\[\tilde{f}(a \delta_s) \tilde{f}(b \delta_t) = (f(a) \delta_s) (f(b) \delta_t) = \phi_{2, s}(\phi_{2, s^{-1}}(f(a))f(b)) \delta_{st}\] which is equal to $\tilde{f}((a\delta_s)(b\delta_t))$.
		
		The fact that the grading is preserved is immediate by the definition, so we are done.

	\end{proof}
	
	\begin{definition}
		We say that $(\mathcal B_1, \Phi_1) \subseteq (\mathcal B_2, \Phi_2)$ is a partial subaction on a generalized Boolean algebra if there exists a morphism of partial actions on generalized Boolean algebras $(\mathcal B_1, \Phi_1) \rightarrow (\mathcal B_2, \Phi_2)$ where the underlying map $\mathcal B_1 \rightarrow \mathcal B_2$ is injective.
	\end{definition}
	
	\begin{theorem} \label{subalgebra}
		If $(\mathcal B_1, \Phi_1) \subseteq (\mathcal B_2, \Phi_2)$ then $\mathrm{Lc}(R, \mathcal B_1) \rtimes_{\Phi_1} G \subseteq \mathrm{Lc}(R, \mathcal B_2) \rtimes_{\Phi_2} G$ as a homogenous $G$-graded subalgebra.
	\end{theorem}
	\begin{proof}
		If the underlying map $f: \mathcal B_1 \rightarrow \mathcal B_2$ is injective, then it's easy to see that the induced map \[\tilde{f}: \mathrm{Lc}(R, \mathcal B_1) \rtimes_{\Phi_1} G \rightarrow \mathrm{Lc}(R, \mathcal B_2) \rtimes_{\Phi_2} G\] is injective and furthermore preserves the grading.
	\end{proof}
	
	\section{Meet Lattices and Tight Filters}\label{meetlatticetight}
	We now cover some results of tight filters in the context of meet semilattices, which will play a big role in the upcoming construction.
	
	\subsection{Meet Semilattices and Filters}
	\begin{definition}
		A \textit{meet semilattice} with $0$ is a meet semilattice with a minimal element $0$. There is a natural notion of a morphism in the category of meet semilattices with $0$ as maps that preserve the $0$ element and meet operation.
	\end{definition}
	
	In this paper, we assume that all our meet semilattices $P$ have a $0$ and simply refer to them as meet semilattices.
	
	\begin{definition}
		For any element $x \in P$, we denote $x^{+} = \{y \in P \colon x \leq y\}$ and $x^{-} = \{y \in P \colon y \leq x\}$. This notation can easily be extended to subsets $X \subseteq P$ by $X^{\pm} = \bigcup_{x \in X} x^{\pm}$. We call a set $X \subseteq P$ \textit{closed upwards} if $X^+ = X$ and \textit{closed downwards} if $X^- = X$.
	\end{definition}
	
	\begin{definition}[{\cite[Section~2.2]{Lawson2011NonCommutativeSD}}]	A \textit{filter} $F \subseteq P$ is a non-trivial (not $\emptyset$ and not $P$) subset of $P$ that is closed upwards and such that for any $a, b \in F$ we have that $a \land b \in F$.
		
		$F(P)$ is the set of all filters in $P$. The topology on $F(P)$ is defined to be generated by the sets $U_{x}= \{F \in F(P) \colon x \in F\}$ for all $x \in P$. This topology is Hausdorff and a basis of compact open sets is formed by sets of the form $U_{(x \colon x_1, \ldots, x_n)} = \{F \in F(P) \colon x \in F, x_1 \notin F, \ldots, x_n \notin F\}$ for any $x, x_1, \ldots, x_n \in F$ and $n \geq 0$.
		
		The \textit{ultrafilters} $U(P)$ are the set of filters that are maximal under inclusion. The \textit{tight filters} $T(P)$ are the closure in $F(P)$ of the ultrafilters with the above topology and inherit a subspace topology. We define $V_{(x \colon x_1, \ldots, x_n)}  = U_{(x \colon x_1, \ldots, x_n)} \cap T(P)$. The sets \[\{V_{(x \colon x_1, \ldots, x_n)}\colon n \geq 0 \text{ with } x, x_1, \ldots, x_n \in P\}\] form a basis of compact open sets of $T(P)$.
	\end{definition}
	
	\begin{remark}
		Our notation differs from \cite{Lawson2011NonCommutativeSD}, however the definitions are the same.
	\end{remark}
	\begin{remark}
		Because $T(P)$ has a topology and inherits the Hausdorff property of $F(P)$, we can construct the generalized Boolean algebra $\mathcal T_c(P)$ of compact open sets of $T(P)$. Because $T(P)$ has a basis of compact open sets of the form $V_{(x \colon x_1, \ldots, x_n)}$, elements in $\mathcal T_c(P)$ are the finite unions of such sets. 
	\end{remark}
	
	\begin{lemma} \label{covgen}
		$\mathcal T_c(P)$ is both covered and generated by $\{V_x \colon x \in P\}$ as a generalized Boolean algebra.
	\end{lemma}
	\begin{proof}
		Both of these facts follow from the fact that elements in $\mathcal T_c(P)$ can be written as finite unions of $V_{(x \colon x_1,\ldots, x_n)}$.  Let $U = \bigcup_{i=1}^m V_{(x_i \colon x_{i, 1}, \ldots x_{i, n_i})} \in \mathcal T_c(P)$ be arbitrary. Clearly $U \subseteq \bigcup_{i=1}^m V_{x_i}$ with $V_{x_i}$ in our set so the set is a cover. To prove that $\mathcal T_c(P)$ is generated by our set, note that $V_{(x \colon x_1,\ldots, x_n)} = V_x \setminus \left( \bigcup_{i=1}^n V_{x_i} \right)$ essentially by definition. Hence, we can form all elements in our basis with relative complements and unions on our set and so we can form the remaining elements with finite unions, so we are done.
	\end{proof}

	We next define the notion of a cover for an element $x$ in a meet semilattice $P$. This notion is intimately connected to the notion of a tight filter, which is our main object of interest. Note that this notion of a cover is unrelated to the cover of a generalized Boolean algebra, which was defined previously.
	
	\begin{definition}
		For any $x \in P$, a \textit{finite cover} of $x$ is a set $\{x_1, \ldots, x_n\} \subseteq x^{-}$ such that for all $0 \neq y \in x^{-}$, there exists some $i$ with $x_i \wedge y \neq 0$.
	\end{definition}

	The following theorem will be our main tool for characterizing when a filter is an ultra/tight filter.
	
	\begin{theorem} \label{ultratight} \;
		
		\begin{enumerate}
		\item A filter $\xi$ is an ultrafilter if and only for all $x \in P \setminus \xi$ there exists some $y \in \xi$ such that $y \wedge x = 0$.
		
		\item A filter $\xi$ is a tight filter if and only if for all $x \in \xi$ and every finite cover $\{x_1, \ldots, x_n\}$ of $x$ there exists some $i \in [1, n]$ such that $x_i \in \xi$.
		\end{enumerate}
	\end{theorem}
	\begin{proof} \;
		\begin{enumerate}
			\item \cite[Lemma~2.4]{Lawson2009ANG} 
			\item \cite[Remark~1.4 and Proposition 2.25]{Lawson2011NonCommutativeSD}
		\end{enumerate}
	\end{proof}
	
	\subsection{Maps on Tight Filters Induced by Meet Subsemilattices}
	In this section, we will analyze how the tight filters of meet subsemilattices $P_1 \subseteq P_2$ interact with each other.
	
	\begin{remark}
		When writing $P_1 \subseteq P_2$, we mean the existence of an injective morphism of semilattices $P_1 \hookrightarrow P_2$ that preserves $0$. To avoid confusion of notation, for $x \in P_2$ we define $x^{P_1, +}$ to be elements of $P_1$ larger than $x$ and similarly define $x^{P_2, +}$ to be elements of $P_2$ larger than $x$ (with similar notation for $x^{P_1, -}$ and $x^{P_2, -}$). For $x \in P_1$, we denote $V_x^{P_1} \subseteq T(P_1)$ to be the tight filters of $P_1$ containing $x$ and $V_x^{P_2}$ to be the tight filters of $P_2$ containing $x$. Note that for the notation $x^{P_1, +}$ we may have $x \in P_1$ or $x \in P_2$, but for the notation $V_x^{P_1}$ we must have that $x \in P_1$ for the statement to make sense.
	\end{remark}

	\begin{lemma}
		Let $P_1 \subseteq P_2$. There is the following partially defined map:
		
		\[re^{\ast}: F(P_2) \dashrightarrow F(P_1) \text{ defined by } \xi \mapsto \xi \cap P_1\]
		$\text { where } \text{re}^{\ast}(\xi) \text{ is not defined whenever } \xi \cap P_1 = \emptyset$.
	\end{lemma}
	
	\begin{proof}
		We want to show that for all $\xi \in F(P_2)$, whenever $\xi \cap P_1 \neq \emptyset$ then $\xi \cap P_1$ is a filter of $P_2$.  To show that the resulting set is closed upwards, let $y \in \xi \cap P_1$ and let $z \in P_1$ such that $y \leq z$. Then we have that $y \in \xi$ and that $y \leq z$ in $P_2$ as well. Hence, this must mean that $z \in \xi$ so $z \in \xi \cap P_1$. $\xi$ is obviously closed under meets because the meet operation is preserved from $P_1 \subseteq P_2$. 
		
		No filter in $F(P_2)$ can contain $0 \in P_2$ so $\xi \cap P_1$ cannot contain $0 \in P_1$ and thus $\xi \cap P_1 \neq P_1$. Hence, $\xi \cap P_1 \subseteq P_1$ is a non-empty and not equal to $P_1$ set that is closed upwards in $P_1$ and under meets, so it is a filter of $P_1$.
	\end{proof}

	\begin{definition}
		We say that $P_1 \subseteq P_2$ \textit{preserves finite covers} if for all $x \in P_1$ and finite covers $\{x_i\}_{i = 1}^n \subseteq P_1$ of $x$ in $P_1$, we have that $\{x_i\}_{i=1}^n \subseteq P_2$ is a cover for $x$ in $P_2$. We often just say that $P_1 \subseteq P_2$ preserves finite covers.
	\end{definition}

	\begin{remark}
		To clarify this definition, we emphasize that the condition of $\{x_i\}_{i=1}^n$ being a cover of $x$ in $P_j$ means that for all $0 \neq y \in x^{P_j, -}$ there exists some $i$ such that $x_i \wedge y = 0$. Note that the condition of being a cover in $P_1$ is weaker than being a cover in $P_2$ because $x^{P_1, -} \subseteq x^{P_2, -}$.
	\end{remark}
	
	\begin{definition}
		For two topological spaces $X_1$ and $X_2$, a \textit{partially defined continuous map} $f: X_1 \dashrightarrow X_2$ is defined to be a continuous map $f: U \rightarrow X_2$ for some open set $U \subseteq X_1$.
	\end{definition}

	\begin{theorem} \label{preservecover}
		For $P_1 \subseteq P_2$ preserving finite covers, the following are true.
		
		\begin{enumerate}			
			\item $re^{\ast}$ restricted to $T(P_2)$ is a partially defined continuous map $re: T(P_2) \dashrightarrow T(P_1)$ defined whenever $\xi \in T(P_2)$ satisfies $\xi \cap P_1 \neq \emptyset$
			
			\item The domain of $re$ is $\bigcup_{x \in P_1} V^{P_2}_x$.
			
			\item  $re^{-1}(V^{P_1}_{(x \colon x_1, \ldots, x_n)}) = V^{P_2}_{(x \colon x_1, \ldots, x_n)}$ for any $n \geq 0$ and  $x, x_1, \ldots, x_n \in P_1$.
			
			\item $re^{-1}$ takes open compact sets of $T(P_1)$ to open compact sets of $T(P_2)$.
		\end{enumerate}
	
		We will write $re_{P_1}^{P_2}$ to refer to the restriction $T(P_2) \dashrightarrow T(P_1)$ when it is unclear which meet subsemilattices we are referring to.
	\end{theorem}
	\begin{proof}
		We will prove the statements together.

		We first show that if $P_1 \subseteq P_2$ preserves finite covers, then, whenever $\xi \in T(P_2)$ and $\xi \cap P_1 \neq \emptyset$, we have that $\xi \cap P_1 \in T(P_1)$. Assume that $\xi \in T(P_2)$ is a tight filter. Let $z \in \xi \cap P_1$ and consider some cover $\{z_i\}_{i=1}^n \subseteq P_1$ for $z$ in $P_1$. We then know that by assumption that $\{z_i\}_{i=1}^n$ is a cover for $z$ in $P_2$. Because $\xi$ is a tight filter of $P_2$, this means that for some $i$ we have that $z_i \in \xi$ and hence $z_i \in \xi \cap P_1$. By Theorem~\ref{ultratight}, $\xi \cap P_1$ is a tight filter of $P_1$ so the map $re: T(P_2) \dashrightarrow T(P_1)$ is well-defined which proves part of (1).
		
		We now show (3) \[re^{-1}(V^{P_1}_{(x \colon x_1, \ldots, x_n)}) = V^{P_2}_{(x \colon x_1, \ldots, x_n)}\] Note that $re^{-1}(V^{P_1}_{(x \colon x_1, \ldots, x_n)}) \subseteq V^{P_2}_{(x \colon x_1, \ldots, x_n)}$ because any tight filter $\xi \in T(P_2)$ such that $\xi \cap P_1$ contains $x$ but not $x_1, \ldots, x_n$ clearly has a similar fact hold in $P_2$. In the other direction, let $\xi \in V^{P_2}_{(x \colon x_1, \ldots, x_n)}$. Because $x \in \xi$, we have that $\xi \cap P_1 \neq \emptyset$,  so $\xi \cap P_1$ is a well-defined tight filter and clearly contains both $x$ and doesn't contain $x_1, \ldots, x_n$. Hence, $\xi \in re^{-1}(V^{P_1}_{(x \colon x_1, \ldots, x_n)})$ and so we have proved (3). 
		
		Because $\{V^{P_1}_x\}_{x \in P_1}$ is a cover for $T(P_1)$, from (3) we get that the domain is \[\bigcup_{x \in P_1} re^{-1}(V^{P_1}_x) = \bigcup_{x \in P_1} V^{P_2}_x\] which proves (2).
		
		The map is partial continuous because our domain is open and the preimage of open sets in a basis of $T(P_1)$ are open sets in $T(P_2)$, so we can extend that result to all open sets in $T(P_1)$, which finishes proving (1).
		
		(4) follows from the fact that sets of the form $V^{P_1}_{(x \colon x_1, \ldots, x_n)}$ forms a compact open basis of $T(P_1)$ and $re^{-1}(V^{P_1}_{(x, \colon x_1, \ldots, x_n)})$ is a compact open set of $T(P_2)$.

	\end{proof}
	
	\begin{example}
		As an example, we present a case of $P_1 \subseteq P_2$ not preserving finite covers where there is a tight filter $\xi \in T(P_2)$ such that $\xi \cap P_1 \neq \emptyset$ and $\xi \cap P_1$ is not a tight filter of $P_1$.
		
		Consider the meet semilattice $P_2$ formed by subsets of $\{1, 2\}$ by the sets \[P_2 = \{\emptyset, \{1\}, \{2\}, \{1, 2\}\}\] where meets are just intersections. Now let $P_1 \subseteq P_2$ be formed by \[P_1 = \{\emptyset, \{1\}, \{1, 2\}\}\]
		
		$P_1 \subseteq P_2$ does not preserve finite covers because the $\{\{1\}\}$ forms a finite cover for $\{1, 2\}$ in $P_1$, yet $\{\{1\}\}$ does not form a finite cover for $\{1, 2\}$ in $P_2$ because $\{2\}\leq \{1, 2\}$ and $\{2\} \cap \{1\} = \emptyset$.

		Let $\xi = \{\{2\}, \{1, 2\}\} \in U(P_2) \subseteq T(P_2)$. However, $\xi \cap P_1 = \{\{1, 2\}\}$ which is not a tight filter in $P_1$ because $\{\{1\}\}$ forms a finite cover for $\{1, 2\}$ in $P_1$, yet $\{1\} \notin \xi$.
		
	\end{example}
	
	Recall that we are interested in the generalized Boolean algebra of the compact open sets of $T(P)$, denoted as $\mathcal T_c(P)$. The next theorem connects our previous results to $\mathcal T_c(P)$.
	
	\begin{theorem} \label{pullbackinj}
		Let $P_1 \subseteq P_2$ preserve finite covers. Then the preimages of $re: T(P_2) \dashrightarrow T(P_1)$ induce an injective map of generalized Boolean algebras \[re^{-1}: \mathcal T_c(P_1) \hookrightarrow  \mathcal T_c(P_2)\] with identification \[U \mapsto re^{-1}(U)\]\[V^{P_1}_{(x \colon x_1, \ldots, x_n)} \mapsto V^{P_2}_{(x \colon x_1, \ldots, x_n)}\] for $n \geq 0$ and  $x, x_1, \ldots, x_n \in P_1$.
	\end{theorem}
	\begin{proof}
		The map is well-defined using (1) and (4) in Theorem~\ref{preservecover}. $re^{-1}$ is a map of generalized Boolean algebras by properties of the preimage. The preimage of $V^{P_1}_{(x \colon x_1, \ldots, x_n)}$ is given by (3) in Theorem~\ref{preservecover}.
		
		It remains to prove that $re^{-1}$ is injective. To do this, first note that because $re^{-1}$ is a map of generalized Boolean algebras, we just have to prove that no element $\emptyset \neq V \in \mathcal T_c(P_1)$ satisfies $re^{-1}(V) = \emptyset$.
		
		For any open set $\emptyset \neq V \subseteq T(P_1)$, $V$ must contain an ultrafilter because $T(P_1)$ is the closure of $U(P_1)$ in $F(P_1)$. Hence, to prove that $re^{-1}(V)$ is non-empty, we just need to prove that every ultrafilter in $T(P_1)$ is the image of an element in $T(P_2)$ by the map $re$.
		
		Let $\xi \in U(P_1)$ be an ultrafilter. A standard exercise with Zorn's lemma proves that there exists some $\xi^{P_2} \in U(P_2)$ such that $\xi \subseteq \xi^{P_2}$. We claim that $\xi = \xi^{P_2} \cap P_1$, which will prove our claim.
		
		Because $\xi \subseteq \xi^{P_2}$, we have that $\xi \subseteq \xi^{P_2} \cap P_1$. If the inclusion was proper, then there would be some $x \in (\xi^{P_2} \cap P_1) \setminus \xi$. Because $x \in P_1 \setminus \xi$ and $\xi$ is an ultrafilter in $P_1$, there must be some $y \in \xi$ such that $x \wedge y = 0$ using Theorem~\ref{ultratight}. But note that $x, y \in \xi^{P_2}$ which implies that $x \wedge y = 0 \in \xi^{P_2}$, a contradiction.
	\end{proof}
	
	\begin{remark}
		Because of the above identification, we can write $V_x$ to refer to either $V^{P_1}_x$ or $V^{P_2}_x$ when $P_1 \subseteq P_2$ preserves finite covers. Namely, when we are interpreting $\mathcal T_c(P_1) \subseteq \mathcal T_c(P_2)$, for $x \in P_2$ we use $V_x$ to refer to $V_x^{P_2} \in \mathcal T_c(P_2)$. If $x \in P_1$, then we can refer to either $V^{P_1}_x$ or $V^{P_2}_x$, but in all arguments we can freely switch between the notations because they are identified together.
	\end{remark}
	
	Because the definition of preserving covers can be slightly difficult to prove, we present a few conditions on $P_1 \subseteq P_2$ that are sufficient for preserving finite covers.
	
	\begin{lemma} \label{thmpreservecover}
		If for all $x \in P_2$, either $x^{P_1, +} = \emptyset$ or there exists $0 \neq y \in P_1$ such that for all $0 \neq y' \in y^{P_1, -}$, we have that $y' \wedge x \neq 0$, then $P_1 \subseteq P_2$ preserves finite covers.
	\end{lemma}
	\begin{proof}
		Let $z \in P_1$ and $\{z_i\}$ be a cover of $z$ in $P_1$. Consider $0 \neq x \in z^{P_2, -}$. 
		
		Clearly $x \leq z$ with $z \in P_1$, so $x^{P_1, +} \neq \emptyset$ and thus by the assumption there exists $y \in P_1$ such that all $0 \neq y' \in y^-$ satisfy $y' \wedge x \neq 0$. Note that this also implies that $y \wedge x \neq 0$ and hence \[y \wedge z \neq 0\] because $y \wedge x \neq 0$ and $x \leq z$.
		
		Because $\{z_i\}$ is a cover of $z$ in $P_1$ and $0 \neq y \wedge z \leq z$ with $y \land z \in P_1$, there exists some $i$ such that \[(y \wedge z) \wedge z_i \neq 0 \Rightarrow y \wedge z_i \neq 0\] For this particular $i$, $0 \neq y \wedge z_i \in y^{P_1, -} $ and hence by the assumption we have that \[(y \wedge z_i) \wedge x \neq 0 \Rightarrow z_i \wedge x \neq 0\]
	\end{proof}
	
	\begin{remark}
		For each $x \in P_2$ such that $x^{P_1, +} \neq \emptyset$, in order to prove the condition for Lemma~\ref{thmpreservecover}, it is enough to just show that there exists $0 \neq y \in P_1$ where $y \leq x$. However, we state our condition in more generality for later applications.
	\end{remark}
	
	\begin{corollary}\label{downwardpreserve}
		Let $P_1 \subseteq P_2$ be closed downwards. Then $P_1 \subseteq P_2$ preserves finite covers.
	\end{corollary}
	\begin{proof}
		Using notation from Lemma~\ref{thmpreservecover}, for all $x \in P_2$ if $x^{P_1, +} \neq \emptyset$ then $x \in P_1$ because $P_1 \subseteq P_2$ is closed downwards, so we take $y=x \in P_1$.
	\end{proof}
	
	\begin{corollary} \label{idealcloseddown}
		Let $P_1 \subseteq P_2$ be closed downwards. Then we can view $\mathcal T_c(P_1) \subseteq \mathcal T_c(P_2)$ as a generalized Boolean subalgebra that identifies $V^{P_1}_x \mapsto V^{P_2}_{x}$ for $x \in P_1$.
	\end{corollary}
	\begin{proof}
		Apply Corollary~\ref{downwardpreserve} and Theorem~\ref{pullbackinj}.
	\end{proof}

	Because we want to generate ideals of a generalized Boolean algebra for our construction, we are also interested in when $P_1 \subseteq P_2$ preserving finite covers induces an ideal $\mathcal T_c(P_1) \subseteq \mathcal T_c(P_2)$. 
	
	\begin{definition}
		Let $P_1 \subseteq P_2$ preserve finite covers. Call $P_1 \subseteq P_2$ \textit{tight} if the induced subalgebra $\mathcal T_c(P_1) \subseteq \mathcal T_c(P_2)$ is also an ideal.
	\end{definition}

	We often split up these definitions and say that $P_1 \subseteq P_2$ preserves finite covers and is tight, even though preserving finite covers is inherently included in the definition of tight.
	
	Our next lemma gives a sufficient condition for $P_1 \subseteq P_2$ being tight.
	\begin{lemma} \label{tightsemigroup}
		$P_1 \subseteq P_2$ is tight if $P_1 \subseteq P_2$ preserves finite covers and for all $x \in P_2$ either $x^{P_1, +} = \emptyset$ or there exists some $y \in x^{P_1,+}$ and finite set $\{y_1, \ldots, y_n\} \subseteq y^{P_1, -}$ such $y_i \wedge x = 0$ for all $i \in [1, n]$ and $\{y_1, \ldots, y_n, x\}$ is a cover for $y$ in $P_2$.
	\end{lemma}
	
	\begin{proof}
		
		We want to prove that for all $U \in \mathcal T_c(P_1)$ and $U' \in \mathcal T_c(P_2)$, we have that $U \cap U' \in \mathcal T_c(P_1)$ using the identification $\mathcal T_c(P_1) \subseteq \mathcal T_c(P_2)$ from Theorem~\ref{pullbackinj}. It suffices to prove this for the elements that generate the respective generalized Boolean algebras. By Lemma~\ref{covgen}, these are sets of the form $U = V_{x}$ and $U' = V_{y}$ with any $x \in P_1$ and $y \in P_2$.
		
	   It is easy to see that $V_x \cap V_y = V_{x \wedge y}$ because a tight filter of $P_2$ contains $x$ and $y$ if and only if it contains $x \wedge y$. Hence, we need to prove that \[V_{x \wedge y} \in \mathcal T_c(P_1)\] for any $x \in P_1$ and $y \in P_2$. 
	   
	   Because $x \wedge y \leq x \in P_1$, we have that $(x \wedge y)^{P_1, +} \neq \emptyset$ so by assumption there exists a some $z \in P_1$ such that $x \wedge y \leq z$ and a finite set $\{z_1, \ldots, z_n\} \subseteq z^{P_1, -}$ such that $\{z_1, \ldots, z_n, x \wedge y\}$ is a cover for $z$ in $P_2$ and $z_i \wedge (x \wedge y) = 0$ for all $i$. We will show that \[V_{x \wedge y} = V_{(z \colon z_1, \ldots, z_n)}\] Because $z, z_1, \ldots, z_n \in P_1$, we have that $V_{(z \colon z_1, \ldots, z_n)} \in \mathcal T_c(P_1)$ so this would be sufficient.
		
		Note that $V_{x \wedge y} \subseteq V_{(z \colon z_1, \ldots, z_n)}$ because any tight filter $\xi \in V_{x \wedge y}$ must contain $z$ because $x \wedge y \leq z$ and also cannot contain any $z_i$ because $z_i \wedge (x \wedge y) = 0$. In the other direction, for any tight filter $\xi \in V_{(z \colon z_1, \ldots, z_n)}$, because $\{z_1,\ldots, z_n, x \wedge y\}$ forms a cover of $z$ and $z \in \xi$ with $z_i \notin \xi$, so using the characterization of tight filters in Theorem~\ref{ultratight} we deduce that $x \wedge y \in \xi$ and so $\xi \in V_{x \wedge y}$.
	\end{proof}
	
	\begin{corollary}\label{downwardideal}
		Let $P_1 \subseteq P_2$ be closed downwards. Then $P_1 \subseteq P_2$ preserves finite covers and is tight. Hence, $\mathcal T_c(P_1) \subseteq \mathcal T_c(P_2)$ is well-defined and an ideal.
	\end{corollary}
	
	\begin{proof}
		If $P_1 \subseteq P_2$ is closed downwards, then $P_1 \subseteq P_2$ preserves finite covers and $\mathcal T_c(P_1) \subseteq \mathcal T_c(P_2)$ is well-defined by Corollary~\ref{idealcloseddown}. To prove that $P_1 \subseteq P_2$ is tight, in the notation of Lemma~\ref{tightsemigroup}, for any $x \in P_2$ such that $x^{P_1, +} \neq \emptyset$, we have that $x \in P_1$. Hence, we take $y = x$ and $n = 0$ as any element is a cover of itself, and so we are done.
	\end{proof}
	
	For $P_1 \subseteq P_2$ preserving finite covers and tight, we describe the ideal $\mathcal T_c(P_1) \subseteq \mathcal T_c(P_2)$ as a collection of open compact sets in $T(P_2)$. This will make certain arguments easier in the future. Recall that for an open set $U$ of a topological space $X$, we define $\mathcal O_c(X)(U)$ as the open compact sets of $X$ contained in $U$.
	
	\begin{lemma} \label{openident}
	Let $P_1 \subseteq P_2$ preserve finite covers and be tight. Define the open set $Y = \bigcup_{x \in P_1} V_x \subseteq T(P_2)$. Then, $\mathcal T_c(P_1) = \mathcal O_c(T(P_2))(Y)$.
	\end{lemma}

	\begin{proof}		
		Noting that the identification $\mathcal T_c(P_1) \subseteq \mathcal T_c(P_2)$ is from the preimage of $re^{-1}$, we know that elements $U \in \mathcal T_c(P_1)$ are compact open subsets in the domain of $re$. But the domain of $re$ is exactly $Y$ by (2) in Theorem~\ref{preservecover}, so $U \in \mathcal O_c(T(P_2))(Y)$.
		
		Now let $U \in \mathcal O_c(T(P_2))(Y)$ be a compact open set contained in $Y$. Note that because $U$ is compact and $Y = \bigcup_{x \in P_1} V_{x}$, there is a finite set $\{x_1, \ldots, x_n\} \in P_1$ such that $U \subseteq \bigcup_{i=1}^n V_{x_i}$. Because $V_{x_i} \in \mathcal T_c(P_1)$, we have that $U$ is less than some element in $\mathcal T_c(P_1)$. Thus, as $\mathcal T_c(P_1) \subseteq \mathcal T_c(P_2)$ is an ideal, we have that $U \in \mathcal T_c(P_1)$, so we are done.
	\end{proof}
	
	\begin{corollary} \label{intersectdown}
		Let $P_1$ and $P_2$ be closed downwards subsets of $P$. Then $P_1 \cap P_2 \subseteq P$ is also closed downwards and $\mathcal T_c(P_1) \cap \mathcal T_c(P_2) = \mathcal T_c(P_1 \cap P_2)$ if we view all sets as ideals of $\mathcal T_c(P)$.
	\end{corollary}

	\begin{proof}
		It's easy to see that $P_1 \cap P_2$ is closed downwards if both $P_1$ and $P_2$ are. All sets then preserve finite covers and are tight with respect to $P$ by Corollary~\ref{downwardideal}.
		
		Applying Lemma~\ref{openident}, it suffices to show that \[\left(\bigcup_{x \in P_1} V_{x} \right)\cap \left(\bigcup_{y \in P_2} V_{y}\right) =  \bigcup_{z \in P_1 \cap P_2} V_{z}\]
		
		But this is easy noting that $V_{x} \cap V_{y} = V_{x \wedge y}$ and that $x \wedge y \leq x, y$, and so $x \wedge y \in P_1 \cap P_2$ because the $P_1$ and $P_2$ are closed downwards, so we are done.
	\end{proof}
	
	\section{Partial Actions from Strongly $E^{\ast}$-Unitary Inverse Semigroups} \label{inversegroupaction}
	
	In the $C^{\ast}$-algebra case, one can associate to a strongly $E^{\ast}$-unitary inverse semigroup a partial crossed product that is realized from an action on the tight representations \cite{Milan2008Semigroup, Milan2011ONIS}. This has been made explicit for inverse semigroups arising from combinatorial objects in \cite{Castro2022CalgebrasOG, deCastro2020}. This approach was also shown to work for combinatorial $R$-algebras in \cite{Boava2021LeavittPA}.
	
	Using well-known a partial action on the idempotents $E$ of a strongly $E^{\ast}$-unitary inverse semigroup $S$ with pure grading $\varphi: S^{\times} \rightarrow G$ described in \cite{Milan2008Semigroup}, we build a partial action $(\mathcal B, \Phi)$ of $G$ on a generalized Boolean algebra that is analogous to the topological partial action in \cite{Milan2011ONIS}. By the contents of Section~\ref{booleanaction}, this associates to $(S, \varphi)$ an $R$-algebra $L_R(S, \varphi) \coloneqq \mathrm{Lc}(R,  \mathcal B) \rtimes_{\Phi} G$. 
	
	We also analyze when inverse subsemigroups $S_1 \subseteq S_2$ lead to partial subactions on generalized Boolean algebras and show, as an application, a natural way to unitize the $R$-algebra of a strongly $E^{\ast}$-unitary inverse semigroup.
	
	\subsection{Basic Definitions}
	
	An overview of inverse semigroup theory can be found in the textbook by Lawson \cite{Lawson1998-ez}. We repeat the required definitions here for convenience.
	
	\begin{definition}
		An \textit{inverse semigroup} is a set $S$ equipped with an associative binary operation where every element $s \in S$ has a unique ``inverse'' $s^{\ast}$ such that $s s^{\ast} s = s$ and $s^{\ast}ss^{\ast} = s^{\ast}$. An inverse semigroup $S$ has a $0$ if there exists $0 \in S$ such that $0 \cdot s = 0 = s \cdot 0$ for all $s \in S$.
	\end{definition}

	We will assume from now on that all our inverse semigroups have a $0$. Throughout, let $S$ be an inverse semigroup and let $E$ be its set of idempotents. The following lemma is standard for inverse semigroups and we will use it frequently without mention.
	
	\begin{lemma}
		The elements of $E$ commute with each other. Furthermore, $E$ is naturally a meet semilattice with $e_1 \wedge e_2 = e_1e_2$. This extends to a partial order on $S$ by defining for $x, y \in S$ the order $x \leq y \Leftrightarrow x = ey$ for some $e \in E$.
	\end{lemma}
	
	We now define strongly $E^{\ast}$-unitary inverse semigroups. We denote $S^{\times} \coloneqq S \setminus \{0\}$ and $E^{\times} \coloneqq E \setminus \{0\}$.
	
	\begin{definition}[{\cite[Section~5]{Milan2011ONIS}}] \label{strongunit}
		$S$ is \textit{strongly $E^{\ast}$-unitary} if there exists a group $G$ and a map $\varphi: S^{\times} \rightarrow G$ such that the following holds:
		
		\begin{enumerate}
			\item $\varphi(ab) = \varphi(a)\varphi(b)$ for all $a, b \in S$ such that $ab \neq 0$
			\item $\varphi^{-1}(e) = E^{\times}$
		\end{enumerate}
		
		We call any $\varphi: S^{\times} \rightarrow G$ that satisfies these properties a \textit{(idempotent) pure grading} of $G$ on $S$. Often we refer to the pair $(S, \varphi)$ as a strongly $E^{\ast}$-unitary inverse semigroup.
	\end{definition}
	
	\begin{remark}
		Strongly $E^{\ast}$-unitary inverse semigroups are also called strongly $0\text{-}E$-unitary inverse semigroups.
	\end{remark}
	
	Finally, we mention a well-known theorem that there exists a ``universal group'' for any strongly $E^{\ast}$-unitary inverse semigroup $S$ in the following sense.
	
	\begin{theorem} \label{universalgroup}
		If $S$ is strongly $E^{\ast}$-unitary then there exists a ``universal group'' $G(S)$ and a pure grading $\varphi: S^{\times} \rightarrow G(S)$ such that, for any group $G$ with a pure grading $\sigma: S^{\times} \rightarrow G$, there is a unique group homomorphism $f: G(S) \rightarrow G$ such that $\sigma = f \circ \varphi$.
	\end{theorem}

	\begin{proof}
		This is well-known and a construction of the universal group for a given $S$ can be found in \cite[Section~2.2]{KELLENDONK2004462}.
	\end{proof}
	
	\subsection{Building a Partial Action}
	Throughout, let $(S, \varphi)$ be a strongly $E^{\ast}$-unitary inverse semigroup with a pure grading $\varphi: S^{\times} \rightarrow G$ and let $E$ denote its idempotents. Let $e = 1_G \in G$ be the unit.
	
	For $g \in G$, if $\varphi^{-1}(g) \neq \emptyset$ define \[E_{g} = \{x \in E \colon x \leq s s^{\ast} \text{ for some } s \in S \text{ with } \varphi(s) = g\}\] If $\varphi^{-1}(g)$ is empty, then define $E_{g} = \{0\}$. Note that $E_e = E$ and $E_g \subseteq E$ is closed downwards.
	
	Our next theorems will define a well-known partial action on $S$ \cite{ Kellendonk2004PartialAO, Milan2011ONIS} by $G$. Because of its lengthy calculation, we omit the proof here.
	
	\begin{theorem} \label{idempotentaction}
		There are well-defined isomorphisms of meet semilattices $\phi_g: E_{g^{-1}} \rightarrow E_g$ that take $x \mapsto sxs^{\ast}$ where $s \in S$ is any element for a given $x$ such that $x \leq s^{\ast}s$ and $\varphi(s) = g$. These meet semilattices and morphisms form a partial action on a set in the sense that the following properties hold:
		\begin{enumerate}
			\item $E_e = E$ and $\phi_e = \mathrm{Id}_E$
			\item $\phi_s(E_{s^{-1}} \cap E_t) = E_s \cap E_{st}$ for all $s, t \in G$
			\item $\phi_s(\phi_t(x)) = \phi_{st}(x)$ for $x \in E_{s^{-1}} \cap E_{(st)^{-1}}$ for all $s, t \in G$
		\end{enumerate}
	\end{theorem}

	Because $E_g \subseteq E$ is closed downwards, we can apply Corollary~\ref{downwardideal} to view $\mathcal T_c(E_g) \subseteq \mathcal T_c(E)$ as an ideal. Furthermore, $\phi_g: E_{g^{-1}} \rightarrow E_g$ is an isomorphism of meet semilattices, so there is an natural induced isomorphism between the tight filters $\tilde \phi_g: T(E_{g}) \rightarrow T(E_{g^{-1}})$ that takes $\xi \mapsto \phi^{-1}_g(\xi)$. This then induces an isomorphism between the compact open sets \[\tilde{\phi}_g^{-1}: \mathcal T_c(E_{g^{-1}}) \rightarrow \mathcal T_c(E_{g^{-1}})\] that takes \[V^{E_{g^{-1}}}_{(x \colon x_1, \ldots, x_n)} \mapsto V^{E_g}_{(\phi_g(x) \colon \phi(x_1), \ldots, \phi_g(x_n))}\]
	
	\begin{remark}
		Because $\phi_g$ is an isomorphism, there is no inherent need to take the inverse anywhere. However, we do so to make some diagrams go the ``correct'' direction in some later proofs.
	\end{remark}
	
	Using these constructions, we will construct a partial action on a generalized Boolean algebra associated to a strongly $E^{\ast}$-unitary inverse semigroup $S$. The construction is exactly the same as applying the construction described in Example~\ref{example:stonedualaction} to the topological partial action described in \cite{Milan2011ONIS}. However, we recreate the proof here in the language of generalized Boolean algebras.

	\begin{theorem} \label{partialconstruct} Let $E_g$ and $\tilde \phi_g^{-1}$ be defined as above. The following defines a partial action of $G$ on the generalized Boolean algebra $\mathcal T_c(E)$.
		
	\[(\{\mathcal T_c(E_g)\}_{g \in G}, \{\tilde \phi_g^{-1}\}_{g \in G})\]
		
	\end{theorem}

	\begin{proof}
		By the statements preceding the theorem, we can view $\mathcal T_c(E_g)$ as an ideal in $\mathcal T_c(E_g)$. We check the conditions in Definition~\ref{booldef}.
		
		Note that we already have that $\mathcal T_c(E_e) = \mathcal T_c(E)$ from $E_e = E$ and $\tilde \phi_e^{-1}$ is the identity from the fact that $\phi_e$ is the identity so (1) in Definition~\ref{booldef} is satisfied.

		For any $E' \subseteq E$ that is downward closed,  $\mathcal T_c(E')$ is generated by $\{V_x\}_{x \in E'}$ so any generalized Boolean algebra morphisms can be checked on the level of individual elements in $E'$. By Lemma~\ref{openident} and Corollary~\ref{intersectdown}, we have that \[\mathcal T_c(E_s \cap E_t) = \mathcal T_c(E_s) \cap \mathcal T_c(E_t)\] when we view them as elements in $\mathcal T_c(E)$. These two statements reduce the checks of (2) and (3) in Definition~\ref{booldef} to

		\[\phi_s(E_{s^{-1}} \cap E_t) = E_s \cap E_{st}\]
		\[\phi_s(\phi_t(x)) = \phi_{st}(x) \text{ for } x \in E_{t^{-1}} \cap E_{(st)^{-1}}\] which are exactly the properties (2) and (3) of Theorem~\ref{idempotentaction}.

	\end{proof}

	Using the results of Section~\ref{booleanaction} and the above Theorem~\ref{partialconstruct}, we have the following definition.
	
	\begin{definition}
		Let $(\mathcal B, \Phi)$ be the partial action of $G$ on the generalized Boolean algebra $\mathcal B = \mathcal T_c(E)$ induced by $(S, \varphi)$. Let $L_R(S, \varphi) \coloneqq \mathrm{Lc}(R, \mathcal B)\rtimes_{\Phi} G$.
	\end{definition} 

	Using the universal group, we now show that the pure grading $\varphi: S^{\times} \rightarrow G$ can be dropped from the notation. 
	
	\begin{theorem} \label{isoswitch}
		Let $S$ be a strongly $E^{\ast}$-unitary inverse semigroup graded by a group $H$ with a pure grading $\sigma: S^{\times} \rightarrow H$. Let $G(S)$ be the universal group of $S$ and let $\varphi: S^{\times} \rightarrow G(S)$ be the associated pure grading. Then $L_R(S, \sigma) \cong L_R(S, \varphi)$.
	\end{theorem}

		\begin{proof}
		Because of the length of the checks needed to rigorously prove every statement, we give mostly an outline.
			
		We will refer to the universal group $G(S)$ as $G$.	By Theorem~\ref{universalgroup}, there exists a group homomorphism $f: G \rightarrow H$ such that $\sigma = f \circ \varphi$. For $g \in G$, let $E_g^{\varphi}$ denote our construction with the pure grading $\varphi$. Similarly, for $h \in H$, let $E_h^{\sigma}$ denote our construction with the pure grading $\varphi$. For any $h \in G$, define \[X_h = \{g \in G \colon f(g) = h\}\]
		
		We first show that for any $h \in H$, the following equality holds.
		
		\begin{equation} \label{eqn:equality} (E_h^{\sigma})^{\times} = \bigsqcup_{g \in X_h} (E_g^{\varphi})^{\times} \end{equation}
		
		Let $0 \neq x \in E^{\sigma}_h$, which implies that there exists $s \in S$ such that  $x \leq s s^{\ast}$ and $\sigma(s) = h$. Note that we must have that $x \in (E^{\varphi}_{\varphi(s)})^{\times}$ because $x \leq ss^{\ast}$ is not affected by the chosen grading. Then, because $\sigma(s) = f(\varphi(s)) \Rightarrow h = f(\varphi(s))$, we find that $\varphi(s) \in X_h$ so $(E_h^{\sigma})^{\times} \subseteq \bigsqcup_{g \in X_h} (E_g^{\varphi})^{\times}$. The other direction is similar.
		
		It remains to prove that our union is disjoint. We do this by showing that for any $g_1 \neq g_2 \in X_h$, we have that $E^{\varphi}_{g_1} \cap E^{\varphi}_{g_2} = \{0\}$. Let $x \in E^{\varphi}_{g_1} \cap E^{\varphi}_{g_2}$. This means that there exists $s_1, s_2 \in S$ such that $x \leq s_1 s_1^{\ast}, s_2 s_2^{\ast}$ and $\varphi(s_1) = g_1$ and $\varphi(s_2) = g_2$. Note that $x^2 = x \leq s_1s_1^{\ast}s_2s_2^{\ast}$. 
		
		If $s_1^{\ast}s_2 = 0$, then $x = 0$ and we are done. Otherwise, $s_1^{\ast}s_2 \neq 0$ and thus we calculate that
		
		\[\varphi(s_1^{\ast}s_2) = g_1^{-1}g_2 \neq 1_G\]
		\[\sigma(s_1^{\ast}s_2) = h^{-1}h = 1_H\]
		
		But $\sigma^{-1}(1_H) = E^{\times} = \varphi^{-1}(1_H)$, hence $s_1^{\ast}s_2$ is both an idempotent and not an idempotent, which is a contradiction.
		
		Note that because $f$ is a group homomorphism, we also have that $X_{h^{-1}} = \{x^{-1} \colon x \in X_h\}$ so 
		
		\[(E_{h^{-1}}^{\sigma})^{\times} = \bigsqcup_{g \in X_h} (E_{g^{-1}}^{\varphi})^{\times}\] 
		
		It's not hard to see then that the partial action $\phi^{\sigma}_h: E^{\sigma}_{h^{-1}} \rightarrow E^{\sigma}_h$, is exactly the action of individual actions of $\phi^{\varphi}_g: E_{g^{-1}}^{\varphi} \rightarrow E_g^{\varphi}$ acting on each individual component for $g \in X_h$.
		
		We next show that \[\mathcal T_c(E^{\sigma}_h) = \bigoplus_{g \in X_h} \mathcal T_c(E^{\varphi}_g)\] where $\bigoplus_{i \in I} \mathcal B_i$ is the generalized Boolean algebra formed by the finite disjoint sums of its components.
		
		We do this using Lemma~\ref{openident}. The lemma says that $\mathcal T_c(E^{\sigma}_h)$ is the open compact subsets of $T(E)$ contained in $\bigcup_{x \in E^{\sigma}_h} V_x$. Similarly, for every $g \in X_h$, $\mathcal T_c(E^{\varphi}_g)$ is the open compact subsets of $T(E)$ contained in $\bigcup_{x \in E^{\varphi}_g} V_x$. It suffices to show that \[\bigcup_{x \in E^{\sigma}_h} V_x = \bigsqcup_{g \in X_h} \left(\bigcup_{x \in E^{\varphi}_g} V_x \right)\]
		
		It's clear that these the unions are equal because $E^{\sigma}_h = \bigcup_{g \in X_h} E^{\varphi}_g$. The fact that the union is disjoint follows directly from Equation \ref{eqn:equality} because if $x \in E^{\varphi}_{g_1}$ and $y \in E^{\varphi}_{g_2}$ for $g_1 \neq g_2$, we have that $xy \in E^{\varphi}_{g_1} \cap E^{\varphi}_{g_2} = \{0\}$ (by downward closure) so $xy = 0$. This means that $V_x \cap V_y = V_{xy} = \emptyset$ so the pairwise intersections of $\bigcup_{x \in E^{\varphi}_g} V_x$ for $g \in X_h$ are disjoint.
		
		Similarly, the actions $(\tilde{\phi}^{\sigma}_h)^{-1}: \mathcal T_c(E^{\sigma}_{h^{-1}}) \rightarrow \mathcal T_c(E^{\sigma}_h)$ are also the action on individual components with $(\tilde{\phi}^{\varphi}_g)^{-1}: \mathcal T_c(E^{\varphi}_{g^{-1}}) \rightarrow \mathcal T_c(E^{\varphi}_g)$ for $g \in X_h$.
		
		Through another similar transformation, we see that \begin{equation}\label{eqn:eqring} \mathrm{Lc}(\mathcal T_c(E^{\sigma}_h), R) \cong \bigoplus_{g \in X_h} \mathrm{Lc}(\mathcal T_c(E^{\varphi}_g), R)\end{equation} where the action on ideals again corresponds to the action on individual components.
		
		From here, we have that the associated partial skew group rings are given as 
		\[L_R(S, \varphi) = \left\{\sum_{g \in G} f_g \delta_g \colon f_g \in \mathrm{Lc}(\mathcal T_c(E^{\varphi}_g), R), f_g = 0 \text{ for all but finitely many } g \right\}\]
		\[L_R(S, \sigma) = \left\{\sum_{h \in H} f_h \delta_h \colon f_h \in \mathrm{Lc}(\mathcal T_c(E^{\varphi}_h), R), f_h = 0 \text{ for all but finitely many } h \right\}\]
		
		There are isomorphisms $L_R(S, \varphi) \rightarrow L_R(S, \sigma)$ as \[f_g \delta_g \mapsto f_g \delta_{f(g)}\] and $L_R(S, \sigma) \rightarrow L_R(S, \varphi)$ as \[f_h \delta_h \mapsto \sum_{g \in X_h} (f_h)_g\delta_g\] where $(f_h)_g$ for $g \in X_h$ is the component associated to $f_h$ in the identification \ref{eqn:eqring}.
	\end{proof}
	
	The following corollary now follows.
	
	\begin{corollary}
	Let $S$ be a strongly $E^{\ast}$-unitary inverse semigroup and let $\varphi_1: S^{\times} \rightarrow G_1$ and $\varphi: S^{\times} \rightarrow G_2$ be two pure gradings on $S$. Then $L_R(S, \varphi_1) \cong L_R(S, \varphi_2)$. Namely, we can just write $L_R(S)$.
	\end{corollary}
	
	Hence, for a strongly $E^{\ast}$-unitary inverse semigroup, we can compute our associated $R$-algebra using any group $G$ and pure grading $\varphi: S^{\times} \rightarrow G$. However, there are many cases where it is helpful to keep in mind a particular group $G$ and pure grading $\varphi: S^{\times} \rightarrow G$. Hence, when convenient, we still write $(S, \varphi)$ and $L_R(S, \varphi)$.
	
	We now show that our induced $R$-algebra $L_R(S)$ is isomorphic to multiple previously considered constructions in the literature. 
	
	\begin{theorem} \label{thm:isogroupoid}
		For a strongly $E^{\ast}$-unitary inverse semigroup $S$, let $\mathcal G_{tight}(S)$ be the usual groupoid of tight representations (\cite{Exel2007InverseSA}) on $S$ and let $A_R(\mathcal G_{tight}(S))$ be the Steinberg algebra (\cite{STEINBERG2010689}). Let $L_R(S)$ be the partial skew group ring defined above. Then,
		
		\[L_R(S) \cong A_R(\mathcal G_{tight}(S))\]
	\end{theorem}

	\begin{proof}
		Let $G$ be the universal group of $S$ and let $E$ be the idempotents of $S$. By \cite[Theorem~5.3]{Milan2011ONIS}, there is a partial action of $G$ on the tight filters $T(E)$. Using standard techniques, this induces a partial skew group ring $\mathrm{Lc}(T(E), R) \rtimes G$ and a transformation groupoid $G \rtimes T(E)$.
		
		By \cite[Theorem~5.3]{Milan2011ONIS},
		
		\[\mathcal G_{tight}(S) \cong G \rtimes T(E)\] as a topological groupoid. 
		
		Note that $T(E)$ is locally compact, Hausdorff, and zero-dimensional, so the partial transformation groupoid $G \rtimes T(E)$ is a Hausdorff ample groupoid. By \cite[Theorem~5.10]{CORDEIRO2020917}, we find that \[\mathrm{Lc}(T(E), R) \rtimes G \cong A_R(G \rtimes T(E))\]
		
		From the outline in Remark~\ref{fulldual}, we have that \[\mathrm{Lc}(T(E), R) \rtimes G \cong \mathrm{Lc}(R, \mathcal O_c(T(E))) \rtimes G\]
		
		Hence, it remains to prove that the partial action on a generalized Boolean algebra that we construct from $S$ is the same as the one induced from the compact open sets of the topological partial action in \cite{Milan2011ONIS} $(\{X_g\}_{g \in G}, \{\theta_g\}_{g \in G})$.
		
		We now briefly describe the necessary facts from the construction, which can be found in \cite[Section~3]{Milan2011ONIS}. For clarity, we note that Section~3 describes the partial action on $F(E)$, however the partial action on $T(E)$ is simply the restriction of this action, of which a discussion can be found in \cite[Section~5]{Milan2011ONIS}. Here, we view tight filters $\xi \in T(E)$ as functions from $E \rightarrow \{0, 1\}$.
		
		\[\beta_s: V_{s^{\ast}s} \rightarrow V_{ss^{\ast}} \text{ with } \beta_s(\xi)(e) = \xi(s^{\ast}es)\]
		\[\theta_g: \bigcup_{s \in \varphi^{-1}(g)} \beta_s \]
		\[X_g = \bigcup_{s \in \varphi^{-1}(g)} V_{ss^{\ast}}\]
		
		We want to show that $(\{\mathcal O_c(X_g), \mathcal O_c(\theta_g)\})$ (as in Example~\ref{example:stonedualaction}) is isomorphic to our generalized Boolean algebra $(\{\mathcal T_c(E_g)\}_{g \in G}, \{\tilde\phi_g^{-1}\})$ in Theorem~\ref{partialconstruct}.
		
		We first show that the open compact subsets of $X_g$ correspond to $\mathcal T_c(E_g)$. By Lemma~\ref{openident}, the ideal of open compact sets of $\mathcal T_c(E_g)$ correspond exactly to $\bigcup_{x \in E_g} V_x$. By definition, each $x \in E_g$ satisfies that $x \leq ss^{\ast}$ for some $s$ such that $\varphi(s) = g$ so for such $x$ we have that $V_x \subseteq V_{ss^{\ast}}$. Furthermore, any such $s$ satisfying this property has that $ss^{\ast} \in E_g$ by definition and hence $\bigcup_{x \in E_g} V_x = \bigcup_{s \in \varphi^{-1}(g)} V_{ss^{\ast}} = X_g$ so we are identifying the same sets. 
		
		It remains to show that the partial action $\phi_g: \mathcal T_c(E_{g^{-1}}) \rightarrow \mathcal T_c(E_{g})$ is the induced one from the topological partial action. Note that because $\mathcal T_c(E_{g^{-1}})$ is generated by $\{V_x\}_{x \in E_{g^{-1}}}$ so it suffices to show that maps agree on these sets. Let $x \in E_{g^{-1}}$ and let $s$ be such that $\varphi(s) = g$ and $x \leq s^{\ast}s$. We have from our construction that $(\tilde \phi_g)^{-1}(V_x) = V_{\phi_g(x)} = V_{sxs^{\ast}}$. For any $\xi \in V_{s^{\ast}s}$, we have that $\beta_s(\xi)(sxs^{\ast}) = \xi(s^{\ast} sx s^{\ast}s) = \xi(x)$. Hence, we have that $\beta_s(\xi)$ contains $sxs^{\ast}$ if and only if $\xi$ contains $x$. Note that $V_x \subseteq V_{s^{\ast}s}$ so we find that $\beta_s(V_x) = V_{sxs^{\ast}}$, which is exactly what we wanted.  
		
		Thus, we have the following isomorphisms, \[L_R(S) \cong \mathrm{Lc}(T(E), R) \rtimes G \cong A_R(\mathcal G_{tight}(S))\]
	\end{proof}
	
	We now simplify notation for elements in our $R$-algebra. For any $g \in G$ and $x \in E_g$, we define $x \delta_g \coloneqq  V_x \delta_g = 1_{V_x} \delta_g \in L_R(S, \varphi)$.
	
	\begin{lemma} \label{computationinverse} Below is the analog for the calculations in Lemma~\ref{computationskew}.
		\begin{enumerate}
			\item $L_R(S)$ is generated as an $R$-algebra by elements of the form $\{x\delta_g\}$ for arbitrary $g \in G$ and $x \in E_g$
			\item $1_{V_x} 1_{V_y} = 1_{V_{xy}}$ for $x, y \in E$
			\item $\tilde{\phi}_{g^{-1}}(1_{V_x}) = 1_{V_{\phi_{g^{-1}}(x)}}$ for $x \in E_g$
			\item $(x \delta_g)(y \delta_{g'}) = \phi_g(y \phi_{g^{-1}}(x))$ for $x \in E_g$ and $y \in E_{g'}$
			\item $(x \delta_e)(y \delta_g) = (xy) \delta_g$ for $x \in E$ and $y \in E_g$
			\item $(y \delta_g) (x \delta_e) = \phi_g(x\phi_{g^{-1}}(y)) \delta_g$ for $x \in E$ and $y \in E_g$
			\item $(x' \delta_e)(y \delta_g)(x \delta_e) =  (x'\phi_g(x \phi_{g^{-1}}(y))) \delta_g$ for $x, x' \in E$ and $y \in E_g$
			\item $(x \delta_g)(y \delta_e)(\phi_{g^{-1}}(x)\delta_{g^{-1}}) = \phi_g(y\phi_{g^{-1}}(x))\delta_e$ for $g \in G$, $y \in E$, and $x \in E_g$
			\item $(x \delta_g)(\phi_{g^{-1}}(x)\delta_{g^{-1}}) = x\delta_e$ for $g \in G$ and $x \in E_g$
			\item Let $x \in E_g$ and $\{x_1, \ldots, x_n\} \subseteq E_g$ be a finite cover of $x$ in $E_g$. Then \[x \delta_g = \sum_{\emptyset \neq J \subseteq [1, \ldots, n]} (-1)^{|J|-1} \left(\prod_{j \in J}  x_j \right) \delta_g\]
			\item $x \delta_g \neq 0$ for any $g \in G$ and $0 \neq x \in E_g$
		\end{enumerate}
	\end{lemma}
	
	\begin{proof}
		Proofs easily reduce to the Lemma~\ref{computationskew} except for (1), (10), and (11).
		
		In (1) the difference is that we are generating now as an $R$-algebra instead of just an $R$-span. Let $A_R$ be the algebra generated by our elements. Recall that $\{V_x\}_{x \in \mathcal B}$ generates the topology of $\mathcal T_c(E)$. Thus, for any $U \in \mathcal B$, we have that $U \delta_e \in A_R$. Similar to the proof of Theorem~\ref{saturatedgeneratingset}, for $U \in \mathcal \mathcal T_c(E_g)$ we have that $U \delta_g$ can be summed up by inclusion exclusion terms, which can be formed by left multiplication of our cover by elements in grading $e$. Now we are back in the case of (1) in Lemma~\ref{computationskew}, so we are done. 
		
		To prove (10), note that if $\{x_1, \ldots, x_n\}$ is a finite cover of $x_i$, we have that $V_x = \bigcup_{i=1}^n V_{x_i}$ because $\{x_1, \ldots, x_n\}$ is a finite cover of $x$ so any filter containing $x$ must contain at least one of the $x_i$. This then also reduces to (10) in Lemma~\ref{computationskew}.
		
		(11) is clear from the standard fact that any $0 \neq x \in E_g$ is contained in an ultrafilter, so $V_x \neq \emptyset$.
	\end{proof}

	Towards an analog of Theorem~\ref{saturatedgeneratingset}, let $S$ be purely graded by a free group $\mathbb F[\mathcal A]$ for some set of generators $\mathcal A$. We present conditions on $(S, \varphi)$ that characterize when the associated partial action $(\mathcal B, \Phi)$ is orthogonal and semi-saturated. 
	
	\begin{theorem} Let $(S, \varphi)$ be a strongly $E^{\ast}$-unitary inverse semigroup $(S, \varphi)$ and let $(\mathcal B, \Phi)$ be the associated partial action. Then,
		\begin{enumerate}
			\item $(\mathcal B, \Phi)$ is \textit{orthogonal} if $E_{a} \cap E_{b} = \{0\}$ for $a \neq b \in \mathcal A$.
			\item $(\mathcal B, \Phi)$ is \textit{semi-saturated} if $E_{st} \subseteq E_t$ for $s, t \in \mathbb F[\mathcal A]$ with $|s+t| = |s| + |t|$
		\end{enumerate}
		We call $(S, \varphi)$ orthogonal/semi-saturated if the respective conditions hold.
	\end{theorem}
	\begin{proof}We prove each condition separately.
		\begin{enumerate}
			\item Note that $\mathcal I_a$ has a cover by $\{\mathcal V_{x}\}_{x \in E_a}$ and similarly for $\mathcal I_b$. Because $\mathcal I_a$ and $\mathcal I_b$ are ideals, it suffices to prove that the pairwise intersection of items in the cover is empty. We know that $V_{x} \cap V_{y} = V_{xy}$.  Because $x \in E_a$ and $y \in E_b$, we know that $xy \in E_a \cap E_b = \{0\}$, so $V_{xy}$ is the empty-set. Hence, $\mathcal I_a \cap \mathcal I_b = \emptyset$ is empty, so the associated partial action on a generalized Boolean algebra is orthogonal.
			
			\item Let $s, t \in \mathbb F[\mathcal A]$ be such that $|s+t| = |s|+|t|$. We will show that $\mathcal T_c(E_{st}) \subseteq \mathcal T_c(E_s)$ (when viewed in $\mathcal T_c(E)$). Lemma~\ref{openident} says that $\mathcal T_c(E_s)$ is the set of all open compact sets contained in $\bigcup_{x \in E_s} V_{x}$ while $\mathcal T_c(E_{st})$ is the set of all compact sets contained in $\bigcup_{x \in E_{st}} V_{x}$. If $E_{st} \subseteq E_t$, then clearly $\bigcup_{x \in E_{st}} V_{x} \subseteq \bigcup_{x \in E_s} V_{x}$, so $\mathcal T_c(E_{st}) \subseteq \mathcal T_c(E_s)$.
		\end{enumerate}
	\end{proof}

	\begin{corollary} \label{saturatedsemigroup} Let $(S, \varphi)$ semi-saturated. For all $a \in \mathcal A$, let $C_a \subseteq E_a$ be sets such that for all $x \in E_{a}$ we have for some $y \in C_a$ that $x \leq y$. For all $a \in \mathcal A$, let $C_{a^{-1}} \subseteq E_{a^{-1}}$ be sets that satisfies an analogous condition. Then a set of generators for $L_R(S, \varphi)$ is \[\{x\delta_e\}_{x \in E} \cup \{x \delta_a\}_{ a\in \mathcal A, x \in C_a} \cup \{x \delta_{a^{-1}}\}_{a \in \mathcal A, x \in C_{a^{-1}}}\]
	\end{corollary}
	
	\begin{proof}
		Because $(S, \varphi)$ semi-saturated implies that the induced partial action on a generalized Boolean algebra is semi-saturated, we can use Theorem~\ref{saturatedgeneratingset}.
		
		We know that $\{V_x \colon x \in E\}$ generates $\mathcal T_c(E)$. Thus, in the notation of Theorem~\ref{saturatedgeneratingset}, we set $C = \{V_x \colon x \in E\}$.
		
		It then suffices to show that, for any $a \in \mathcal A$, $\{V_x \colon x \in C_a\}$ covers $\mathcal T_c(E_a)$ as a generalized Boolean algebra (the analogous statement will hold for $a^{-1}$). Let $U \in \mathcal T_c(E_a)$ be viewed as a compact set in $T(E_a)$. For all tight filters $\xi \in U$, there is an element $x \in \xi$. By the assumption, there exists $y \in C_a$ such that $x \leq y$ so we have that $\xi \in V_y \in \mathcal T_c(E_x)$. Hence, the union of all $\{V_x \colon x \in C_a\}$ contains $U$ as a set. Because $U$ is compact, there is a finite such cover. Hence, $\{V_x\}_{x \in C_a}$ forms a cover for the generalized Boolean algebra $\mathcal T_c(E_a)$.
		
		A similar proof works for $a^{-1}$, so we are done.
	\end{proof}
	
	Finally, we will look at inverse subsemigroups. We write $S_1 \subseteq S_2$, to refer to an injective homomorphism $f: S_1 \rightarrow S_2$ that preserves the $0$ and \textbf{all} multiplications, not just non-zero ones. Note that $S_1 \subseteq S_2$ induces a natural inclusion $E_1 \subseteq E_2$.
	
	Furthermore,  if $S_2$ is a strongly $E^{\ast}$-unitary inverse semigroup with a pure grading $\varphi: S_2^{\times} \rightarrow G$ for some group $G$ and $S_1 \subseteq S_2$, it's not difficult to prove the induced map $\varphi_{\mid S_1^{\times}}: S_1^{\times} \rightarrow G$ is a pure grading on $S_1$ and hence $S_1$ is a strongly $E^{\ast}$-unitary inverse semigroup. 
	
	From now on, let $S_2$ be a strongly $E^{\ast}$-unitary inverse semigroup graded by group $G$ with pure grading $\varphi: S_2^{\times} \rightarrow G$ and let $S_1$ be an inverse semigroup such that $S_1 \subseteq S_2$ where $S_1$ is given the pure grading from $\varphi$. We refer to the pure grading on $S_1$ as $\varphi$, even though more precisely it should be $\varphi_{\mid S_1^{\times}}$. For the strongly $E^{\ast}$-inverse semigroups $S_i$, let $E_{i, g}$ be our constructed downwards closed meet semilattices, $\phi_{i, g}: E_{i, g^{-1}} \rightarrow E_{i, g}$ be the meet semilattice isomorphisms, and $(\mathcal B_i, \Phi_i)$ be our associated partial action on a generalized Boolean algebra.
	
	 In general, $S_1 \subseteq S_2$ not enough to induce a partial subaction $(\mathcal B_1, \Phi_1) \subseteq (\mathcal B_2, \Phi_2)$. To remedy this, we define a slightly stronger notion of an inverse subsemigroup $S_1 \subseteq_c S_2$.
	
	\begin{definition} \label{subinverse} Let $S_1$ and $S_2$ be strongly $E^{\ast}$-unitary inverse semigroups. We say that \[S_1  \subseteq_c S_2\] if the following hold:
		\begin{enumerate}
			\item $S_1 \subseteq S_2$
			\item $E_1 \subseteq E_2$ (by the identification $S_1 \subseteq S_2$) preserves finite covers
		\end{enumerate}
	\end{definition}

	\begin{theorem}
	If $S_1 \subseteq_c S_2$, then there is an induced partial subaction $(\mathcal B_1, \Phi_1) \subseteq (\mathcal B_2, \Phi_2)$ and thus $L_R(S_1, \varphi) \subseteq L_R(S_2, \varphi)$ as a homogenous $G$-graded $R$-subalgebra where the identification $L_R(S_1, \varphi) \hookrightarrow L_R(S_2, \varphi)$ is given by the ``identity'' \[x \delta_g \mapsto x \delta_g\] for $g \in G$ and $x \in E_{1, g}$
	\end{theorem}
	
	\begin{proof}
		Because $E_1 \subseteq E_2$ preserves finite covers, we have a map $\mathcal T_c(E_1) \hookrightarrow \mathcal T_c(E_2)$ by $(re^{E_2}_{E_1})^{-1}$ using Theorem~\ref{pullbackinj}.

		Our goal is to prove that this map is actually a map of partial actions on the generalized Boolean algebras defined in Definition~\ref{def:booleanmorph}. We repeat the definition in our current context for convenience.  Recall that for $i = 1, 2$ and $g \in G$, we are interpreting $\mathcal T_c(E_{1, g}) \subseteq \mathcal T_c(E_1)$ as the image of $(re_{E_{1, g}}^{E_2})^{-1}$.
		
		\begin{enumerate}
			\item $(re_{E_1}^{E_2})^{-1}(\mathcal T_c(E_{1, g})) \subseteq \mathcal T_c(E_{2, g})$
			\item The following commutative diagram holds
			
			$\begin{tikzcd}
				\mathcal T_c(E_{1, g^{-1}}) \arrow[d, "\tilde \phi_{1, g}", swap] \arrow[r, "(re_{E_1}^{E_2})^{-1}"] & \mathcal T_c(E_{2, g^{-1}}) \arrow[d, "\tilde \phi_{2, g}"] \\
				\mathcal T_c(E_{1, g}) \arrow[r, "(re_{E_1}^{E_2})^{-1}", swap] & \mathcal T_c(E_{2, g}) \\
			\end{tikzcd}$
		\end{enumerate}
		
		We first check that $(re_{E_1}^{E_2})^{-1}(\mathcal T_c(E_{1, g})) \subseteq \mathcal T_c(E_{2, g})$. To do this, we first show that $E_1 \hookrightarrow E_2$ induces an injection $E_{1, g} \hookrightarrow E_{2, g}$. Let $x \in E_{1, g}$. Then we have that $x \leq ss^{\ast}$ for some $s \in S_1$ with $\varphi(s) = g$. Because the grading on $S_1$ is induced by the one on $S_2$, the same is true in $S_2$ as well, so $x \in E_{2, g}$.
		
		It's easy then to check that the following diagram commutes.
		
		\[\begin{tikzcd}
			E_1 \arrow[r, hook] & E_2  \\
			E_{1, g}\arrow[r, hook,  swap] \arrow[u, hook] & E_{2, g} \arrow[u,hook, swap]\\
		\end{tikzcd}\]
		
		\begin{lemma}
			For meet semilattices $P_1$ and $P_2$, if $P_1 \subseteq P_2$ preserves finite covers then for any $P'_1 \subseteq P_2$ that preserves finite covers and $P'_2 \subseteq P_2$ such that $P'_1 \subseteq P'_2$, we have that $P'_1 \subseteq P'_2$ preserves finite covers.
		\end{lemma}
		
		\begin{proof}
			By viewing $P'_1 \subseteq P_1 \subseteq P_2$, all finite covers in $P'_1$ are finite covers in $P_1$, which are then finite covers in $P_2$, so $P'_1 \subseteq P_2$ preserves finite covers. All finite covers in $P_2$ containing elements only in $P'_2$ are also covers when viewed in the smaller $P'_2$. Hence, because $P'_1 \subseteq P'_2$, we have that $P'_1 \subseteq U_2$ preserves finite covers.
		\end{proof}

		Using the previous lemma and applying Lemma~\ref{downwardpreserve}, we have that all injections in our diagram preserve finite covers. Hence, we have the following commutative diagrams (where the second diagram is induced by Theorem~\ref{pullbackinj}).
		
		\[\begin{tikzcd}
			T(E_1) \arrow[d, dashed, "re_{E_{1, g}}^{E_1}", swap] & \arrow[l, dashed, "re_{E_1}^{E_2}", swap] \arrow[d,dashed, "re_{E_{2, g}}^{E_2}"] T(E_2))  \\
			T(E_{1, g}) & \arrow[l, dashed, "re_{E_{1, g}}^{E_{2, g}}"] T(E_{2, g}) \\
		\end{tikzcd} \qquad \begin{tikzcd}
		\mathcal T_c(E_1) \arrow[r, hook, "(re_{E_1}^{E_2})^{-1}"] & \mathcal T_c(E_2)  \\
		\mathcal T_c(E_{1, g}) \arrow[r, hook, "(re_{E_{1, g}}^{E_{2, g}})^{-1}", swap] \arrow[u, hook, "(re_{E_{1, g}}^{E_1})^{-1}"] & \mathcal T_c(E_{2, g}) \arrow[u,hook, "(re_{E_{2, g}}^{E_2})^{-1}", swap]\\
	\end{tikzcd}\]
		
		Hence, $(re_{E_1}^{E_2})^{-1}$ takes $\mathcal T_c(E_{1, g})$ to $\mathcal T_c(E_{2, g})$, so we are done. Note that this also proves that the map $(re_{E_1}^{E_2})^{-1}$ on $\mathcal T_c(E_{1, g})$ is induced by the inclusion $E_{1, g} \hookrightarrow E_{2, g}$.
		
		We now prove that our desired diagram commutes. By the same process as before, we have the following commutative diagrams. The first diagram commutes because for any $g \in G$ and $x \in E_{1, g^{-1}}$, for any $s \in S_1$ such $\varphi(s) = g$ and $x \leq s^{\ast}s$, a similar fact holds in $S_2$ so we have that $\phi_{1, g}(x) = s x s^{\ast} = \phi_{2, g}(x)$.
		
		\[\begin{tikzcd}
			E_{1, g^{-1}} \arrow[r, hook] \arrow[d, "\phi_{1, g}", swap] & E_{2, g^{-1}} \arrow[d, "\phi_{2, g}"] \\
			E_{1, g} \arrow[r, hook] & E_{2, g} \\
		\end{tikzcd}\qquad 		
	\begin{tikzcd}
		T(E_{1, g^{-1}})  & T(E_{2, g^{-1}})  \arrow[l, "re_{E_{1, g^{-1}}}^{E_{2, g^{-1}}}", swap, dashed] \\
		T(E_{1, g})  \arrow[u, "\tilde \phi_{1, g}"] & T(E_{2, g}) \arrow[l, "re_{E_{1, g}}^{E_{2, g}}", dashed] \arrow[u, "\tilde \phi_{2, g}", swap] \\
	\end{tikzcd} \qquad
	\begin{tikzcd}
	\mathcal T_c(E_{1, g^{-1}}) \arrow[r, "(re_{E_{1, g^{-1}}}^{E_{2, g^{-1}}})^{-1}", hook] \arrow[d, "\tilde \phi_{1, g}^{-1}", swap] & \mathcal T_c(E_{2, g^{-1}}) \arrow[d, "\tilde \phi_{2, g}^{-1}"] \\
	\mathcal T_c(E_{1, g}) \arrow[r, "(re_{E_{1, g}}^{E_{2, g}})^{-1}", swap, hook] & \mathcal T_c(E_{2, g}) \\
\end{tikzcd}\]

	The last diagram is exactly our desired diagram. Hence, $(re_{E_1}^{E_2})^{-1}$ induces a morphism $(\mathcal B_1, \Phi_1) \rightarrow (\mathcal B_2, \Phi_2)$.  Because $(re_{E_1}^{E_2})^{-1}$ is injective, we have that $(\mathcal B_1, \Phi_1) \subseteq (\mathcal B_2, \Phi_2)$ and hence $L_R(S_1, \varphi) \subseteq L_R(S_2, \varphi)$.
	
	By Theorem~\ref{preservecover}, the map$(re_{E_1}^{E_2})^{-1}: \mathcal T_c(E_1) \rightarrow \mathcal T_c(E_2)$ takes $V_x^{E_1}\mapsto V_{x}^{E_2}$ for $x \in E_1$ and hence the map on $R$-algebras sends $x \delta_g \mapsto x\delta_g$, so we are done.
	\end{proof}
	
	\subsection{Unitization}
	As a brief application, we develop a notion of unitization on a strongly $E^{\ast}$-unitary inverse semigroup $S$ that corresponds to a notion of unitization on $L_R(S)$.
	
	\begin{definition}
		For an inverse semigroup $S$, define the unitization of $S$ to be a new inverse semigroup $S^{\ast}$ where  \[S^{\ast} = S \cup \{\ast\}\] and operations are defined with $\ast$ as an identity and other operations are inherited from $S$.
	\end{definition}

	Let $E^{\ast}$ denote the idempotents of $S^{\ast}$. Note that any $s \in S$ satisfies that $s \leq \ast$ because $\ast \in E^{\ast}$ and $s = \ast \cdot s$. If $S$ can be purely graded by $G$ with morphism $\varphi: S^{\times} \rightarrow G$, then it's easy to show that $\varphi^{\ast}: (S^{\ast})^{\times} \rightarrow G$ is a pure grading on $S^{\ast}$ by $g$ where \[\varphi^{\ast}(s) = \varphi(s)\] for any $0 \neq s \in S$ and \[\varphi^{\ast}(\ast) = 1_G\] Hence, $S^{\ast}$ is also strongly $E^{\ast}$-unitary.

	\begin{lemma} \label{unitize}
		$L_R(S^{\ast}, \varphi^{\ast})$ is a unitary $R$-algebra.
	\end{lemma}
	\begin{proof}
		By Corollary~\ref{unitalcor}, it suffices to prove that the generalized Boolean algebra $\mathcal T_c(E^{\ast})$ has a unit. $V_{\ast}$ contains all tight filters of $E^{\ast}$ as $x \leq \ast$ for any $x \in E^{\ast}$, so $V_{\ast} = T(E^{\ast})$.
	\end{proof}
	
	We now show that our unitization $L_R(S^{\ast}, \varphi^{\ast})$ has some nice properties with respect to $L_R(S, \varphi)$. These will be proven when $S \neq 0$ is not the trivial inverse semigroup (it's arguable whether such a group is strongly $E^{\ast}$-unitary). Note that $S \neq 0$ implies that $E \neq 0$ because $ss^{\ast} \in E$ and for any $s \neq 0$ we have that $ss^{\ast}s = s \Rightarrow ss^{\ast} \neq 0$.
	
	\begin{theorem} \label{unitizethm}
		Let $(S, \varphi) \neq 0$ be a strongly $E^{\ast}$-unitary inverse semigroup. Then $S \subseteq_c S^{\ast}$ so there is an induced subalgebra $L_R(S, \varphi) \subseteq L_R(S^{\ast}, \varphi)$. Furthermore, $L_R(S, \varphi) \subseteq L_R(S^{\ast}, \varphi^{\ast})$ satisfies that $L_R(S, \varphi)= L_R(S^{\ast}, \varphi^{\ast})$ if and only if $L_R(S, \varphi)$ is unital.
	\end{theorem}
	
	\begin{proof}
		We first show that $(S, \varphi) \subseteq_c (S^{\ast}, \varphi^{\ast})$ as in Definition~\ref{subinverse}.
		
		The inclusion $S \subseteq S^{\ast}$ is trivial. $E \subseteq E^{\ast}$ preserves finite covers because $E \subseteq E^{\ast}$ is downwards closed and so we can apply Corollary~\ref{downwardpreserve}. Thus, there is an induced subalgebra $L_R(S, \varphi) \subseteq L_R(S^{\ast}, S^{\ast})$. 
		
		If $L_R(S, \varphi)$ is not unital, then clearly $L_R(S, \varphi) \neq L_R(S^{\ast}, \varphi^{\ast})$ as $L_R(S^{\ast}, \varphi^{\ast})$ is unital by Lemma~\ref{unitize}.
		
		In the other direction, assume that $L_R(S, \varphi)$ is unital. We will show that the induced partial actions on their generalized Boolean algebras are actually the same. First note that \[E^{\ast}_g = \{s^{\ast}s \colon s \in S, \varphi^{\ast}(s) = g^{-1} \}^-\]
		If $g \neq e$, then $\varphi^{\ast}(s) = g^{-1}$ implies that $s \in S$. Because $s^{\ast}s \in S$ as well and $\ast$ is a maximal element in $S^{\ast}$, we get that $E^{\ast}_g = E_g$ for $g \neq e$. Hence, $\mathcal T_c(E^{\ast}_g) \cong \mathcal T_c(E_g)$.
		
		For the case $g=e$, we must prove that the subalgebra $ \mathcal T_c(E) \subseteq \mathcal T_c(E \cup \{\ast\})$ induced by $E \subseteq E^{\ast}$ is actually the whole set. We know that it is actually an ideal because $E \subseteq E^{\ast}$ is downwards closed. Hence, it suffices to prove that the unit $V_{\ast} \in \mathcal T_c(E)$.
		
		If $L_R(S, \varphi)$ is unital, by Corollary~\ref{unitalcor} this implies that $\mathcal T_c(E)$ contains a unit. Namely, the topological space $T(E)$ is compact. Because $\{V^E_x\}_{x \in E}$ covers $T(E)$ and $T(E)$ is compact, there is some finite set $0 \neq x_1, \ldots, x_n \in E$ with $n \geq 1$ (by $E \neq 0$) such that $\bigcup_{i=1}^n V^E_{x_i} = T(E)$. We will prove that $\{x_i\}_{i=1}^n$ forms a finite cover of $\ast$ in $E \cup \{\ast\}$. This suffices because it will imply that \[\bigcup_{i=1}^n V^{E^{\ast}}_{x_i} = V_{\ast}\] with $V^{E^{\ast}}_{x_i} \in \mathcal T_c(E)$.
		
		Let $0 \neq x \leq \ast$. If $ x = \ast$, then $x_1 \cdot \ast = x_1 \neq 0$. If $x \neq \ast$, then $x \in E$. Because $x$ is non-zero, there exists some ultrafilter $\xi$ in $E$ containing $x$. Because $\bigcup_{i=1}^n V^E_{x_i} = T(E)$, we have that there is some $i$ such that $\xi \in V_{x_i}$. With this $i$, we have that $x_i, x \in \xi$ and thus $x_i \wedge x \neq 0$, so we are done.
	\end{proof}
	
	\begin{definition}
		An ideal $I \subseteq R$ is an \textit{essential ideal} if for all $x \in R$, we have that $xI = Ix = 0 \Rightarrow x = 0$.
	\end{definition}

	\begin{theorem}
		$L_R(S, \varphi) \subseteq L_R(S^{\ast}, \varphi^{\ast})$ is an essential ideal.
	\end{theorem}

	\begin{proof}
		We first prove that $L_R(S, \varphi) \subseteq L_R(S^{\ast}, \varphi^{\ast})$ is an ideal. 
		
		By Lemma~\ref{computationinverse} (1), we know that $L_R(S^{\ast}, \varphi^{\ast})$ is generated as an $R$-algebra by \[\{x\delta_g\}_{g \in G, x \in E^{\ast}_g}\] Using Theorem~\ref{localunits} and the proof of Lemma~\ref{computationinverse} (1), we find that $L_R(S, \varphi)$ has a set of local units generated as a ring by $\{x\delta_e\}_{x \in E}$.
		
		It thus suffices to show that for all $g \in G, x \in E^{\ast}_g$ and $y \in E$, we have that \[(y \delta_e)(x \delta_g), (x\delta_g)(y\delta_e)\in L_R(S, \varphi)\]
		
		As in the proof of Theorem~\ref{unitizethm}, we have that $\mathcal T_c(E_g) \cong \mathcal T_c(E_g)$ for $g \neq e$. Namely, $L_R(S, \varphi)$ contains all homogenous elements $L_R(S^{\ast}, \varphi^{\ast})$ with grading $g \neq e$. Hence, if $g \neq e$, then $(y \delta_e)(x \delta_g), (x\delta_g)(y\delta_e)$ are either $0$ or in the grading $g$, thus are clearly in $L_R(S, \varphi)$. Thus, we assume that $g = e$. In this case, $(y\delta_e)(x\delta_e) = (x \delta_e)(y \delta_e) = (xy) \delta_e \in L_R(S, \varphi)$ clearly because $x \in E \Rightarrow xy \in E$.
		
		It remains to prove that $L_R(S, \varphi) \subseteq L_R(S^{\ast}, \varphi^{\ast})$ is essential. Let $x \in L_R(S^{\ast}, \varphi^{\ast})$ such that $L_R(S, \varphi)x = 0$. By a standard grading argument, we may assume without loss of generality that $x = f \delta_g$ for some $0 \neq f \in \mathrm{Lc}(R, \mathcal T_c(E^{\ast}_g))$. If $g \neq e$, then similarly $x \in L_R(S, \varphi)$ as proven earlier and because $L_R(S, \varphi)$ has local units $L_R(S, \varphi)x$ cannot equal $0$. Hence, again we can assume that $g = e$.
		
		Now write $f = \sum_{i=0}^{n} r_i U_i \delta_e$ for pairwise disjoint $\{U_i\}_{i=1}^n \in \mathcal T_c(E^{\ast})$. Note that for $V \in \mathcal T_c(E)$, we have that $V \delta_e \in L_R(S, \varphi)$ satisfies that $(V \delta_e)x = \sum_{i=1}^n r_i (U_i \wedge V) \delta_e$. Because the $U_i$ were disjoint, this is clearly non-zero if and only if at least one of the $i$ satisfies that $U_i \wedge V \neq 0$.
		
		Hence, it suffices to prove that, for any $\emptyset \neq U \in \mathcal T_c(E^{\ast})$, there exists some $V \in \mathcal T_c(E)$ such that $U \wedge V \neq 0$. Recall that the identification is given by the preimage of the partial map \[T(E^{\ast}) \dashrightarrow T(E)\] given by \[\xi \mapsto \xi \cap E \text{ whenever $\xi \cap E \neq \emptyset$}\]
		
		Using this, it's clear that $U$ has non-empty intersection with some $V \in \mathcal T_c(E)$ if there is some filter $\xi \in U$ such that $\xi \cap E \neq \emptyset$. Recall from the proof of Theorem~\ref{pullbackinj} that $U$ must contain an ultrafilter $\xi$ of $E^{\ast}$. Because $E^{\ast} \setminus E = \{\ast\}$, the only filter in $T(E^{\ast})$ with empty intersection with $E$ is $\{\ast\}$. As $E \neq \emptyset$, it's easy to see using Theorem~\ref{ultratight} that $\{\ast\}$ cannot be an ultrafilter. Hence, any ultrafilters in $T(E^{\ast})$ must have non-empty intersection with $E$, and so we are done.
		
	\end{proof}

	Finally, to summarize, we remove mentions of the pure grading.
	
	\begin{theorem}
		Let $S$ be a strongly $E^{\ast}$-unitary inverse semigroup. Let $S^{\ast}$ denote the inverse semigroup $S \cup \{\ast\}$ where $\ast$ is an added unit. Then $S^{\ast}$ is a strongly $E^{\ast}$-unitary inverse semigroup and $S \subseteq_c S^{\ast}$. Furthermore, the following holds:
		
		\begin{enumerate}
			\item $L_R(S^{\ast})$ is a unital $R$-algebra
			\item There is a natural inclusion $L_R(S) \subseteq L_R(S^{\ast})$
			\item $L_R(S) = L_R(S^{\ast})$ if and only if $L_R(S)$ is unital
			\item $L_R(S) \subseteq L_R(S^{\ast})$ is an essential ideal
		\end{enumerate}
	\end{theorem}

	\section{Proofs of Isomorphism}\label{isomorphism}
	
	We now prove that Leavitt path algebras and labelled Leavitt path algebras can be realized as the algebra of a strongly $E^{\ast}$-unitary inverse semigroup associated to each combinatorial object. Note that the fact that the algebras are isomorphic can be proven by applying Theorem~\ref{thm:isogroupoid} and using known realizations of these algebras as Steinberg algebras. However, we apply our construction here in detail as reference for future calculations.
	
	\subsection{Leavitt Path Algebras}
	
	In this section, we only describe the inverse semigroup with respect to our construction. We keep the section for reference, however we omit the proofs as they involve the same ideas as the labelled Leavitt path algebra case found in Section~\ref{labelledleavittiso}. All definitions for Leavitt path algebras can be found in the textbook by Abrams \cite{Abrams2017}.
	
	\begin{definition}[Directed Graph Inverse Semigroups  \cite{Ash1975}]
		For a graph $G = (E^0, E^1)$, we define an inverse semigroup \[S_G \coloneqq  \{(p_1, p_2) \in E^{\ast} \times E^{\ast} \colon r(p_1) = s(p_2)\} \cup \{0\}\] where $E^{\ast}$ denotes the set of finite length paths (possibly just a vertex) in $G$.
		
		Multiplication is defined as 
		
		\[(a_1, a_2)\cdot(b_1, b_2) = \begin{cases}
			(a_1, b_2) & \text{ if } a_2 = b_1 \\
			(a_1b_1', b_2) &\text{ if } b_1 = a_2b_1' \\
			(a_1, b_2a_2') & \text{ if } a_2 = b_1a_2' \\
			0 & \text{otherwise}
		\end{cases}\]
		
		with $(a, b)^{\ast} = (b, a)$ and the set of idempotents
		
		\[E = \{(p, p) \in S\} \cup \{0\}\]
		
		The meet semilattice of idempotents has a natural order \[(p_1, p_1) \leq (p_2, p_2) \Leftrightarrow p_2 \text{ is a prefix of } p_1\]
		
		The inverse semigroup has a pure grading $\varphi: S_G^{\times} \rightarrow \mathbb F[E^1]$ defined by taking \[(p_1, p_2) \mapsto p_1 p_2^{-1}\] where we view a path $p = e_1e_2\ldots e_n \in \mathbb F[E^1]$ and vertices as the empty word.
		
		In what follows, whenever we represent $g \in \mathbb F[E^1]$ we assume that the representations are reduced. Furthermore, assume that $g \neq e$ as that information is built into the definition.

		The sets $\varphi^{-1}(g)$ are:
		
		\[\varphi^{-1}(g) = \begin{cases} 
			\{(p_1p, p_2p) \colon  s(p) = r(p_1) \} \cup \{0\} & \text{ if }g = p_1 p_2^{-1} \text{ where } p_1, p_2 \in E^{\geq 1} \text{ and } r(p_1) = r(p_2) \\
			\{(p_1p, p) \colon r(p_1) = s(p) \} \cup \{0\}& \text{ if }g = p_1 \text{ where } p_1 \in E^{\geq 1} \\
			\{(p, p_1p) \colon s(p) = r(p_1)\} \cup \{0\} & \text{ if }g = p_1^{-1} \text{ where } p_1 \in E^{\geq 1}\\
			\{0\} & \text{otherwise}
		\end{cases}\]

		The sets $E_g$ are:
		
		\[E_g = \begin{cases} 
			\{(p, p) \colon p_1 \text { is a prefix of } p \} \cup \{0\} & \text{ if }g = p_1 p_2^{-1} \text{ where } p_1, p_2 \in E^{\geq 1} \text{ and } r(p_1) = r(p_2) \\
			\{(p, p) \colon p_1 \text { is a prefix of } p\} \cup \{0\}& \text{ if }g = p_1 \text{ where } p_1 \in E^{\geq 1} \\
			\{(p, p) \colon s(p) = r(p_1)\} \cup \{0\} & \text{ if }g = p_1^{-1} \text{ where } p_1 \in E^{\geq 1}\\
			\{0\} & \text{otherwise}
		\end{cases}\]
		
		With this pure grading, $(S, \varphi)$ orthogonal and semi-saturated.
		
		The maps $\phi_{g^{-1}}: E(S)_{g} \rightarrow E(S)_{g^{-1}}$ act in the following manner:
		
		\[\phi_{g^{-1}}(p, p) = \begin{cases} 
			(pp', pp') \text{ with } p = p_1p' & \text{ if }g = p_1 p_2^{-1} \text{ where } |p_1|, |p_2| \geq 1 \\
			(p', p') \text { with } p= p_1p'  & \text{ if }g = p_1 \text{ where } |p_1| \geq 1 \\
			(p_1p, p_1p) & \text{ if } g = p_1^{-1} \text{ where } |p_1| \geq 1 \\
		\end{cases}\]

		The isomorphism from the Leavitt Algebra $L_R(G)$ with generators $\{v\}_{v \in E^0} \cup \{s_p, s_p^{\ast}\}_{p \in E^1}$ to the $R$-algebra $L_R(S_G, \varphi)$ is given by the map on generators\[v \mapsto(v, v) \delta_e, s_p \mapsto (p, p)\delta_p, \text{ and } s^{\ast}_p \mapsto (r(p), r(p))\delta_{p^{-1}}\]
		
	\end{definition}
	
	\subsection{Labelled Leavitt Path Algebras}	\label{labelledleavittiso}
	We first define labelled spaces and their associated labelled Leavitt Path algebras. All definitions can be found in \cite{Boava2021LeavittPA}.
	
	A labelled graph is a directed graph $\mathcal E = (\mathcal E^0, \mathcal E^1)$ paired with a labeling $\mathcal L: \mathcal E^1 \rightarrow \mathcal A$ where $\mathcal A$ is an alphabet. Labelling obviously extends to finite length paths in $\mathcal E^{\ast}$.
	
	We denote $\mathcal L^{\ast} = \mathcal L(\mathcal E^{\ast})$. For some subset $A \subseteq \mathcal E^0$ and some $\alpha \in \mathcal L^{\ast}$, we denote $r(A, \alpha)$ to be the range of paths $p = e_1\ldots e_{|\alpha|}$. More precisely, $r(A, \alpha) = \{r(p) \colon p \in \mathcal E^{\ast}, \mathcal L(p) = \alpha\}$. We define $r(\alpha) = r(\mathcal E^0, \alpha)$.
	
	Let $\mathcal B$ be a set of subsets of $\mathcal E^0$ closed under finite unions and intersections that contains $r(\alpha)$ for all $\alpha \in \mathcal L^{\ast}$ and such that for all $A \in \mathcal B$ and $\alpha \in \mathcal L^{\ast}$ we have that $r(A, \alpha) \in \mathcal B$.
	
	We call a triplet $(\mathcal E, \mathcal L, \mathcal B)$ is a \textit{labelled space}.
	
	A labelled space is \textit{weakly left-resolving} if for all $A, B \in \mathcal B$ and $\alpha \in \mathcal L^{\ast}$, we have that $r(A \cap B, \alpha) = r(A, \alpha) \cap r(B, \alpha)$. A labelled space is \textit{normal} if $\mathcal B$ is closed under relative complements. We assume that all labelled spaces considered henceforth are weakly left-resolving and normal and just refer to them as labelled spaces.
	
	For $\alpha \in \mathcal L^{\ast}$, we define $\mathcal B_{\alpha} = \mathcal B \cap \mathcal P(r(\alpha))$. Define $\mathcal B_{\omega} = \mathcal B$. $\mathcal B_{\alpha}$ is an ideal in of $\mathcal B$.
	
	For $A \in \mathcal B$, define $\Delta_A = \{a \in \mathcal A \colon r(A, a) \neq \emptyset\}$. Let $\mathcal E_{\text{sink}}$ be the sinks of $\mathcal E$ (as a directed graph). Note that $\mathcal E_{\text{sink}}$ is not necessarily in $\mathcal B$.
	
	For $A \in \mathcal B$, call $A$ be \textit{regular} if $0 < |\Delta_A| < \infty$ and there exists no $\emptyset \neq B \in \mathcal B$ such that $B \subseteq A \cap \mathcal E_{\text{sink}}$. If $A$ is not regular, we refer to it as \textit{singular}. We denote $\mathcal B_{\text{reg}} \subseteq \mathcal B$ to be the regular sets. $\mathcal B_{\text{reg}}$ is an ideal of $\mathcal B$.
	
	The associated labelled Leavitt path algebra $L_R(\mathcal E, \mathcal L, \mathcal B)$ is defined as the associative $R$-algebra on the generators $\{p_A\}_{A \in \mathcal B} \cup \{s_a, s_a^{\ast}\}_{a \in \mathcal A}$ that satisfy the following relations, which are often known as the Cuntz-Krieger relations.
	
	\begin{enumerate}
		\item $p_{A \cap B} = p_A p_B$, $p_{A \cup B} = p_{A} + p_B - p_{A \cap B}$, $p_{\emptyset} = 0$ for $A, B \in \mathcal B$
		\item $p_A s_a = s_a p_{r(A, a)}$ and $s_a^{\ast}p_A = p_{r(A, a)}s_a^{\ast}$ for $A \in \mathcal B$ and $a \in \mathcal A$
		\item $s_a^{\ast}s_b = \delta_{a, b} p_{r(a)}$ for $a, b \in \mathcal A$
		\item $s_a s_a^{\ast} s_a = s_a$ and $s_a^{\ast}s_as_a^{\ast} = s_a^{\ast}$ for $a \in \mathcal L$.
		\item For all $A \in \mathcal B_{\text{reg}}$, $p_A = \sum_{a \in \Delta_A} s_a p_{r(A, a)}s_a^{\ast}$
	\end{enumerate}

	\begin{definition}[Labelled Leavitt Path Algebra Semigroups]
		
		For a normal weakly left-resolving labelled space $(\mathcal E, \mathcal L, \mathcal B)$ define its inverse semigroup (\cite{Boava17Inverse}) \[S_{(\mathcal E, \mathcal L, \mathcal B)} \coloneqq \{(\alpha, A, \beta) \colon \alpha, \beta \in \mathcal L^{\ast} \text{ and } \emptyset \neq A \in \mathcal \mathcal B_{\alpha} \cap \mathcal \mathcal B_{\beta}\}\cup \{0\}\]
		
		Multiplication is defined as 
		
		\[(\alpha, A, \beta)\cdot(\gamma, B, \delta) = \begin{cases}
			(\alpha, A \cap B, \delta) & \text{ if } \beta = \gamma \\
			(\alpha\gamma', r(A, \gamma') \cap B, \delta) &\text{ if } \gamma = \beta \gamma' \text{ and } r(A, \gamma') \cap B \neq \emptyset \\
			(\alpha, A \cap r(B, \beta'), \delta \beta') & \text{ if } \beta = \gamma \beta' \text{ and } A \cap r(B, \beta') \neq \emptyset \\
			0 & \text{otherwise}
		\end{cases}\]
		with $(\alpha, A, \beta)^{\ast} = (\beta, A, \alpha)$ and the set of idempotents 
		
		\[E = \{(\alpha, A, \alpha) \colon \alpha \in \mathcal L^{\ast} \text{ and } \emptyset \neq A \in \mathcal B_{\alpha}\}\cup\{0\}\]
		The meet semilattice of idempotents has a natural order 
		
		\[(\alpha, A, \alpha) \leq (\beta, B, \beta) \Leftrightarrow \alpha = \beta \alpha' \text { and } A \subseteq r(B, \alpha')\]
		
		The inverse semigroup has a pure grading $\varphi: S_{(\mathcal E, \mathcal L, \mathcal B)}^{\times} \rightarrow \mathbb F[\mathcal A]$ defined by taking \[(\alpha, A, \beta) \mapsto \alpha \beta^{-1}\] where we view $\alpha = \alpha_1 \alpha_2\ldots \alpha_n \in \mathbb F[\mathcal A]$. 
		
		In the following, we assume that representations for $g \in \mathbb F[\mathcal A]$ are reduced.
		
		The sets $\varphi^{-1}(g)$ are:
		\[\varphi^{-1}(g) = \begin{cases} 
			\{(p_1p, A, p_2p) \colon  p \in \mathcal L^{\ast}, \emptyset \neq A \in \mathcal B_{p_1p} \cap \mathcal B_{p_2p}\}& \text{ if } g = p_1 p_2^{-1} \text{ where } p_1, p_2 \in \mathcal L^{\ast} \\
			\emptyset & \text{otherwise}
		\end{cases}\]
		
		The sets $E_{g}$ are:
		
		\[E_{g} = \begin{cases} 
			\{(p_1p, A, p_1p) \colon  p \in \mathcal L^{\ast}, \emptyset \neq A \in \mathcal B_{p_1p} \cap \mathcal B_{p_2p}\} \cup \{0\} & \text{ if } g = p_1 p_2^{-1} \text{ where } p_1, p_2 \in \mathcal L^{\ast} \\
			\{0\} & \text{otherwise}
		\end{cases}\]
		
		The maps $\phi_{g^{-1}}: E_g \rightarrow E_{g^{-1}}$ act in the following manner for $g = p_1p_2^{-1}$:
		
		\[\phi_{g^{-1}}(p_1p, A, p_1p) = (p_2p, A, p_2p)\]
		
		With this pure grading, $(S, \varphi)$ is orthogonal and semi-saturated.
		
		The isomorphism from the labelled Leavitt path algebra $L_R(\mathcal E, \mathcal L, \mathcal B)$ with generators $\{p_B\}_{B \in \mathcal B} \cup \{s_{a}, s_{a}^{\ast}\}_{a \in \mathcal A}$ to $L_R(S_{\mathcal E, \mathcal L, \mathcal B})$ is given by the map on generators 		\[p_B \mapsto (\omega, B, \omega)\delta_e, s_{a} \mapsto (a, r(a), a)\delta_a, \text{ and } s^{\ast}_a \mapsto (\omega, r(a), \omega)\delta_{a^{-1}}\]
	\end{definition}
	
	\begin{lemma}
		$S_{(\mathcal E, \mathcal L, \mathcal B)}$ is an inverse semigroup with the given idempotents and order.
	\end{lemma}
	\begin{proof}
		{\cite[Section~3]{Boava17Inverse}}
	\end{proof}
	
	\begin{lemma}
		There is a pure grading $\varphi: S_{(\mathcal E, \mathcal L, \mathcal B)}^{\times} \rightarrow \mathbb F[\mathcal A]$ defined by \[(\alpha, A, \beta) \mapsto \alpha \beta^{-1}\]
	\end{lemma}
	\begin{proof}
		We need to prove that $\varphi(ab) = \varphi(a)\varphi(b)$ when $ab \neq 0$ and that $\varphi^{-1}(\omega) = E^{\times}$
		
		Let $(\alpha, A, \beta), (\gamma, B, \beta) \in S$. We will do casework on the multiplication.
		
		Assume that $(\alpha, A, \beta) \cdot (\gamma, B, \delta) \neq 0$. Note that $\varphi(\alpha, A, \beta) = \alpha \beta^{-1}$ and $\varphi(\gamma, B, \delta) = \gamma \delta^{-1}$ Then there are three cases.
		
		If $\beta = \gamma$, then $\varphi((\alpha, A, \beta) \cdot (\gamma, B, \delta)) = \varphi((\alpha, A \cap B, \delta)) = \alpha \delta^{-1} = \alpha \beta^{-1} \gamma =\delta^{-1} = \varphi(\alpha, A, \beta) \varphi(\gamma, B, \delta)$.
		
		If $\gamma = \beta \gamma'$ and $r(A, \gamma') \cap B \neq \emptyset$, $\varphi((\alpha, A, \beta) \cdot (\gamma, B, \delta)) = \varphi((\alpha \gamma', r(A, \gamma'), \delta)) = \alpha \gamma' \delta^{-1} = \alpha \beta^{-1} \beta \gamma' \delta^{-1} = \alpha \beta^{-1} \gamma \delta^{-1}= \varphi(\alpha, A, \beta) \varphi(\gamma, B, \delta)$.
		
		The last case is the same.
		
		To show that $\varphi^{-1}(\omega) = E^{\times}$, note that $(\alpha, A, \beta) \in S^{\times}$ satisfying $\varphi(s) = \omega$ is exactly equivalent to $\alpha \beta^{-1} = \omega$. But this happens if and only if $\alpha = \beta$. $(\alpha, A, \beta) \in E^{\times}$, so we are done.
	\end{proof}
	
	\begin{lemma}
		For $g = p_1p_2^{-1}$ with $p_1, p_2 \in \mathcal L^{\ast}$ and $p_1p_2^{-1}$ reduced, our construction is described as follows:
		
		\[\varphi^{-1}(g) =  \{(p_1p, A, p_2p) \colon  p \in \mathcal L^{\ast}, \emptyset \neq A \in \mathcal B_{p_1p} \cap \mathcal B_{p_2p}\}\]
		\[E_g = 	\{(p_1p, A, p_1p) \colon  p \in \mathcal L^{\ast}, \emptyset \neq A \in \mathcal B_{p_1p} \cap \mathcal B_{p_2p}\} \cup \{0\} \]
		\[\phi_{g^{-1}}: E_g \rightarrow E_{g^{-1}} \text{ takes } (p_1p, A, p_1p) \mapsto (p_2p, A, p_2p)\]
		
		For all other $g$, the resulting elements are trivial.
	\end{lemma}
	
	\begin{proof} Throughout, let $g = p_1p_2^{-1}$ be a reduced representation.
		
		We first show that \[\varphi^{-1}(g) = \{(p_1p, A, p_2p) \colon  p \in \mathcal L^{\ast}, \emptyset \neq A \in \mathcal B_{p_1p} \cap \mathcal B_{p_2p}\}\] First note that when $\emptyset \neq A \in \mathcal \mathcal B_{p_1p} \cap \mathcal B_{p_2p}$, then $(p_1p, A, p_2p)$ is actually an element of $S$. Furthermore, $\varphi(p_1p, p_2p) = (p_1p)(p_2p)^{-1} = p_1p_2^{-1}$ so we are done. In the other direction, if $\varphi((\alpha, A, \beta)) = \alpha \beta^{-1} = p_1p_2^{-1}$, by the reduced nature of $p_1p_2^{-1}$ we would have that $\alpha = p_1p$ and $\beta = p_2p$ for some $p \in \mathcal L^{\ast}$, so we are done.
		
		We now show that 	\[E_g = \{(p_1p, A, p_1p) \colon  p \in \mathcal L^{\ast}, \emptyset \neq A \in \mathcal B_{p_1p} \cap \mathcal B_{p_2p}\} \cup \{0\} \] We write $E_g = \bigcup_{s \in \varphi^{-1}(g)} \{x \in E \colon x \leq ss^{\ast}\}$. By our result in the last paragraph, we calculate that $\{ss^{\ast} \colon s \in \varphi^{-1}(g)\} = \{(p_1p, A, p_1p) \colon \emptyset \neq A \in \mathcal B_{p_1p} \cap \mathcal B_{p_2p}\}$. We will show that this set is closed downwards (besides $0$) so we get that this set with $0$ is actually equal to $E_g$.
		
		Let $(\alpha, B, \alpha) \leq (p_1p, A, p_1p)$. Then by our characterization of the order of idempotents, we have that $\alpha = p_1p\gamma$ and $B \subseteq r(A, \gamma)$. This means that $B \in \mathcal B_{p_1p\gamma} \cap \mathcal B_{p_2p\gamma}$. Hence, we find that $(\alpha, B, \alpha) = (p_1(p\gamma), B, p_1(p\gamma))$ with $B \in \mathcal B_{p_1p\gamma} \cap \mathcal B_{p_2p\gamma}$ which is also in our set, so we are done.
		
		We now show that \[\phi_{g^{-1}}: E_g \rightarrow E_{g^{-1}} \text{ takes } (p_1p, A, p_1p) \mapsto (p_2p, A, p_2p)\]
		
		We already know that $(p_1p, A, p_1p) \in E_g$ means that $A \in \mathcal B_{p_1p} \cap \mathcal B_{p_2p}$. Consider $s = (p_1, r(p_1) \cap r(p_2), p_2) \in \varphi^{-1}(g)$. Then note that $ss^{\ast} = (p_1, r(p_1) \cap r(p_2), p_1)$. We know that $r(r(p_1) \cap r(p_2), p) = r(p_1p) \cap r(p_2p)$ (because the space is weakly left-resolving) which is maximal in $\mathcal B_{p_1p} \cap \mathcal B_{p_2p}$. Hence, $A \subseteq r(r(p_1) \cap r(p_2), p)$ so $(p_1p, A, p_1p) \leq ss^{\ast}$. By our construction, this means that \[\phi_{g^{-1}}(p_1p, A, p_1p) = (p_2, r(p_1) \cap r(p_2), p_1)(p_1p, A, p_1p)(p_1, r(p_1) \cap r(p_2), p_2) = (p_2p, A, p_2p)\] by an easy calculation, so we are done.
	\end{proof}
	
	\begin{lemma}
		The inverse semigroup $(S_{\mathcal E, \mathcal L, \mathcal B}, \varphi)$ is orthogonal and semi-saturated.
	\end{lemma}
	\begin{proof}
		We first prove that it is orthogonal. Let $a \neq b \in \mathcal A$. $E_a \cap E_b = \{0\}$ from the characterization of $E_a$ and $E_b$ because the associated path $\alpha$ for $(\alpha, A, \alpha) \in E_a$ must begin with $a$, while it must begin with $b$ in $E_b$. Hence, the only shared element is $0$.
		
		We now prove that it is semi-saturated. Let $s, t \in \mathbb F[\mathcal A]$ with $|s+t| = |s| + |t|$. We want to prove that $E_{st} \subseteq E_s$. We say $g \in \mathbb F[\mathcal A]$ has \textit{correct form} if its reduced form can be written as $p_1p_2^{-1}$. Namely, if $g$ is not of correct form then $E_g = \{0\}$.
		
		If $s$ or $t$ is not of correct form and $|st| = |s| + |t|$, then $st$ is not of correct form either so $\{0\}= E_{st} \subseteq E_s$ is obvious. Thus, from now on we can assume that $s = p_{1, s}p_{2, s}^{-1}$ and $t = p_{1, t}p_{2, t}^{-1}$. 
		
		If $|s+t| = |s| + |t|$ and $st$ is of correct form, then we must either have that $p_{2, s} = \omega$ or $p_{1, t} = \omega$. The proofs of each case are similar, so we only treat the first.
		
		If $p_{2, s} = \omega$, then we want to prove that $E_{p_{1, s}t} \subseteq E_{p_{1, s}}$. Because $st$ is reduced, we know that $st = (p_{1, s}p_{1, t})p_{2, t}^{-1}$ is reduced. We can show that the maximal element in $E_{p_{1, s}t}$ is $(p_{1, s}p_{1, t}, r(p_{1, s}p_{1, t}) \cap r(p_{2, t}), p_{1, s}p_{1, t})$. This is in in $E_{p_{1, s}}$, so we are done because the maximal element of $E_{p_{1, s}t}$ is in $E_{p_{1, s}}$ and $E_{p_{1, s}t}$ is downwards closed.
	\end{proof}
	
	\begin{lemma}
		There is a morphism from $L_R(\mathcal E, \mathcal L, \mathcal B) \rightarrow L_R(S_{\mathcal E, \mathcal L, \mathcal B})$ given by the maps on generators 
		
		\[p_B \mapsto (\omega, B, \omega)\delta_e, s_{a} \mapsto (a, r(a), a)\delta_a, \text{ and } s^{\ast}_a \mapsto (\omega, r(a), \omega)\delta_{a^{-1}}\]
	\end{lemma}
	
	\begin{proof}
		As $L_R(\mathcal E, \mathcal L, \mathcal B)$ is universal, it suffices to check that the Cuntz-Krieger relations hold in the image. Here, we often make references to Lemma~\ref{computationinverse} without explicitly mentioning it.
		
		For ease of reference, we repeat the Cuntz-Kreiger relations below (found in \cite[Definition~3.1]{Boava2021LeavittPA}).
		
		\begin{enumerate} 
			\item $p_{A \cap B} = p_A p_B$, $p_{A \cup B} = p_{A} + p_B - p_{A \cap B}$, $p_{\emptyset} = 0$ for $A, B \in \mathcal B$
			\item $p_A s_a = s_a p_{r(A, a)}$ and $s_a^{\ast}p_A = p_{r(A, a)}s_a^{\ast}$ for $A \in \mathcal B$ and $a \in \mathcal L$
			\item $s_a^{\ast}s_b = \delta_{a, b} p_{r(a)}$ for $a, b \in \mathcal L$
			\item $s_a s_a^{\ast} s_a = s_a$ and $s_a^{\ast}s_as_a^{\ast} = s_a^{\ast}$ for $a \in \mathcal L$.
			\item For all $A \in \mathcal B_{\text{reg}}$, $p_A = \sum_{a \in \Delta_A} s_a p_{r(A, a)}s_a^{\ast}$
		\end{enumerate}
		
		We now check that each of these relations hold for their associated elements.
		\begin{enumerate}
			\item $p_Ap_B = ((\omega, A, \omega) \delta_e)((\omega, B, \omega))\delta_e = (\omega, A \cap B, \omega) = p_{A \cap B}$. $p_A + p_B = p_{A \cup B}$ holds from noting that $\{(\omega, A, \omega), (\omega, B, \omega)\}$ forms a finite cover for $(\omega, A \cup B, \omega)$ so we can use the inclusion exclusion calculation (10) in Lemma~\ref{computationinverse}.
			
			\item We have that  \[p_A s_a = ((\omega, A, \omega)\delta_e)((a, r(a), a) \delta_a) = ((\omega, A, \omega)(a, r(a),a))\delta_a = (a, r(A, a), a) \delta_a\] \[s_ap_{r(A, a)} = ((a, r(a), a) \delta_a)((\omega, r(A, a), \omega)\delta_e) = \phi_a((\omega, r(A, a), \omega)\phi_{a^{-1}}((a, r(a),a))\delta_a = \]\[\phi_a((\omega,r(A, a), \omega)(\omega, r(a), \omega))  \delta_a = \phi_a((\omega, r(A, a), \omega)) \delta_a = (a, r(A, a), a)\delta_a\] The analogous $s_a^{\ast}p_A = p_{r(A, a)}s_a^{\ast}$ is proved similarly.
			
			\item $s_a^{\ast}s_b = ((\omega, r(a), \omega)\delta_{a^{-1}})((b, r(b), b)\delta_b) = x \delta_{a^{-1}b}$ for some $x \in E_{a^{-1}b}$ and thus must be $0$ if $a \neq b$ because $E_{a^{-1}b} = \{0\}$. If $a = b$, then by noting that $(a, r(a), a) = \phi_{a}(\omega, r(a), \omega)$, we can immediately deduce from Lemma~\ref{computationinverse} (9) that the result is $(\omega, r(a), \omega) \delta_e = p_{r(a)}$.
			
			\item $s_as_a^{\ast}s_a = s_a p_{r(a)}$ by (3). By the same calculations as in (2), we can deduce that $s_ap_{r(a)} = (a, r(a), a) \delta_a = s_a$. The relation $s_a^{\ast}s_as_a^{\ast} = s_a^{\ast}$ is proved in the same way.
			
			\item This will be the most involved check. Let $A \in \mathcal B_{\text{reg}}$. We first calculate $s_ap_{r(A, a)}s_a^{\ast}$ using (8) in Lemma~\ref{computationinverse}.
			
			\[((a, r(a), a)\delta_a)((\omega, r(A, a), \omega)\delta_e)((\omega, r(a), \omega)\delta_{a^{-1}}) = \phi_a((\omega, r(A, a), \omega)\phi_{^{-1}}(a, r(a), a)) \delta_e =\]\[\phi_a((\omega, r(A, a), \omega)(\omega, r(a), \omega)) \delta_e = (a, r(A, a), a) \delta_e\]
			
			Our goal is to show that $\{(a, r(A, a), a)\}_{a \in \Delta_A}$ is a cover for $(\omega, A, \omega)$. Note that $(a, r(A, a), a)(b, r(A, b), b) = 0$ for $a \neq b$. Hence all the products with more than one element would be $0$ in (10) Lemma~\ref{computationinverse}, so we would have \[p_A = (\omega, A, \omega)\delta_e = \sum_{a \in \Delta_a} (a, r(A, a), a) \delta_e = \sum_{a \in \Delta_a} s_ap_{r(A, a)} s_a^{\ast}\] and we would be done.
			
			We have that $(a, r(A, a), a) \leq (\omega, A, \omega)$, so $(a, r(A, a), a) \in (\omega, A, \omega)^-$ so the set can be a cover.
			
			Now let $(\gamma, B, \gamma) \leq (\omega, A, \omega)$. The only condition here is that $\emptyset \neq B \subseteq r(A, \gamma)$. We want to prove that there exists some $a \in \Delta_a$ such that $(\gamma, B, \gamma)(a, r(A, a), a) \neq 0$.
			
			First assume that $|\gamma| = 0 \Rightarrow B \subseteq A$. Because $A$ is regular, $B$ cannot be a subset of $\mathcal E_{\text{sink}}$. Hence, there must be some $a \in \Delta_B \subseteq \Delta_A$ such that $r(B, a) \neq 0$. For this $a \in \Delta_A$, we then calculate that $(\omega, B, \omega)(a, r(A, a), a) = (a, r(A, a) \cap r(B, a), a) = (a, r(B, a), a) \neq 0$, so we are done.
			
			Now assume that $|\gamma| \geq 1$ and consider the first character $\gamma_1 \in \mathcal A$. Note that $\gamma_1 \in \Delta_A$ because $\emptyset \neq B \subseteq r(A, \gamma)$ implies that $r(A, \gamma_1)\neq \emptyset$. In this case, we can see that $(\gamma, B, \gamma) \leq (\gamma_1, r(A, \gamma_1), \gamma_1)$ because $B \subseteq r(A, \gamma)$, so clearly $(\gamma, B, \gamma)(\gamma_1, r(A, \gamma_1), \gamma_1) \neq 0$, so we are done.
		\end{enumerate}
	\end{proof}
	
	\begin{lemma}
		The morphism from $L_R(\mathcal E, \mathcal L, \mathcal B) \rightarrow L_R(S_{(\mathcal E, \mathcal L, \mathcal B)})$ is an isomorphism.
	\end{lemma}
	
	\begin{proof}
		We can view the grading on $L_R(S_{(\mathcal E, \mathcal L, \mathcal B)})$ by $\mathbb F[\mathcal A]$ as a grading on $\mathbb Z$ because there is a natural morphism from $\mathbb F[\mathcal A] \rightarrow \mathbb Z$ taking $a \mapsto 1$ and $a^{-1} \mapsto -1$ for $a \in \mathcal A$. Under this grading by $\mathbb Z$, it's not difficult to check that the above morphism is a $\mathbb Z$-graded morphism. Hence we can use the graded uniqueness theorem \cite[Corollary~5.5]{Boava2021LeavittPA} and (11) Lemma~\ref{computationinverse} to show that the morphism is injective.
		
		To show that the morphism is surjective, we will use Corollary~\ref{saturatedsemigroup}. Note that for $a \in \mathcal A$, we have that $E_a = 0$ or has $\{(a, r(a), a)\}$ as a maximal element. Similarly, $E_{a^{-1}} = 0$ or has $\{(\omega, r(a), \omega)\}$ as a maximal element. Hence, in the notation of Corollary~\ref{saturatedsemigroup}, we can take $C_x = \{(a, r(a), a)\}$ and $C_{x^{-1}} = \{(\omega, r(a), \omega)\}$. These appear as the image of $s_a$ and $s_a^{\ast}$ respectively.
		
		Hence, all that remains to do is to prove that $x\delta_e$ is in the image for all $x \in E$. Let $x = (\alpha, A, \alpha)$ with $\alpha \in \mathcal L^{\ast}$ and $\emptyset \neq A \in \mathcal B_\alpha$. 
		
		If $|\alpha| = 0$, then $x \delta_e$ is the image of $p_A$. 
		
		Now assume that $|\alpha| \geq 1$. We will show that $(\omega, r(\alpha), \omega)\delta_{\alpha^{-1}}$ and $(\alpha, r(\alpha), \alpha) \delta_{\alpha}$ are in the image (whenever well-defined). If we can do this, then for any $A \in \mathcal B_{\alpha}$, we can calculate \[((\alpha, r(\alpha), \alpha)\delta_{\alpha})((\omega, A, \omega)\delta_e)((\omega, r(\alpha), \omega) \delta_{\alpha^{-1}}) = (\alpha, A, \alpha)\delta_e\] by a similar calculation of $s_ap_{r(A, a)}s_a^{\ast}$ in condition (5) of the last theorem.
		
		Hence, it remains to show that $(\omega, r(\alpha), \omega)\delta_{\alpha^{-1}}$ and $(\alpha, r(\alpha), \alpha) \delta_{\alpha}$ are in our image. The case $|\alpha| = 1$ is just the images of $s_a$ and $s_a^{\ast}$, so we can prove this inductively by showing that the following holds (whenever non-zero):
		
		\begin{enumerate}
			\item $((\omega, r(\alpha), \omega)\delta_{\alpha^{-1}}) ((\omega, r(\beta), \omega) \delta_{\beta^{-1}}) = (\omega, r(\beta\alpha), \omega)\delta_{(\beta\alpha)^{-1}}$
			\item 	$((\alpha, r(\alpha), \alpha) \delta_{\alpha})(\beta, r(\beta), \beta) \delta_{\beta} = (\alpha\beta, r(\alpha\beta), \alpha\beta) \delta_{\alpha\beta}$
		\end{enumerate}
		
		(1) can be calculated as follows:
		
		\[((\omega, r(\alpha), \omega)\delta_{\alpha^{-1}}) ((\omega, r(\beta), \omega) \delta_{\beta^{-1}}) = \phi_{\alpha^{-1}}((\omega, r(\beta), \omega)\phi_{\alpha}((\omega, r(\alpha), \omega)))\delta_{(\beta\alpha)^{-1}} = \]
		\[\phi_{\alpha^{-1}}((\omega, r(\beta), \omega)(\alpha, r(\alpha), \alpha))\delta_{(\beta\alpha)^{-1}} = \phi_{\alpha^{-1}}((\alpha, r(\beta\alpha) \cap r(\alpha), \alpha))\delta_{(\beta\alpha)^{-1}} =\] \[(\omega, r(\beta\alpha) \cap r(\alpha), \omega)\delta_{(\beta\alpha)^{-1}} = (\omega, r(\beta\alpha), \omega) \delta_{(\beta\alpha)^{-1}}\]
		
		(2) can be calculated as follows:
		
		\[((\alpha, r(\alpha), \alpha) \delta_{\alpha})(\beta, r(\beta), \beta) \delta_{\beta} = \phi_{\alpha}((\beta, r(\beta), \beta)\phi_{\alpha^{-1}}((\alpha, r(\alpha), \alpha)))\delta_{\alpha\beta} = \] \[ = \phi_{\alpha}((\beta, r(\beta), \beta)(\omega, r(\alpha), \omega))\delta_{\alpha\beta} = \phi_{\alpha}((\beta, r(\alpha\beta) \cap r(\alpha), \beta))\delta_{\alpha\beta} = \]\[(\alpha\beta, r(\alpha\beta) \cap r(\beta), \alpha\beta) \delta_{\alpha\beta} = (\alpha\beta, r(\alpha\beta), \alpha\beta) \delta_{\alpha\beta}\]
		
		Hence, the image of our morphism contains a set of generators for $L_R(S_{(\mathcal E, \mathcal L, \mathcal B)})$, so the map is surjective and we are done.
	\end{proof}
	
	\bibliographystyle{abbrv}
	\bibliography{GeneralizedBooleanAction.bib}

\begin{thebibliography}{10}

\bibitem{Abadie2004OnPA}
F.~Abadie.
\newblock On partial actions and groupoids.
\newblock {\em Proceedings of the American Mathematical Society},
  132(4):1037--1047, 2004.

\bibitem{Abrams2017}
G.~Abrams, P.~Ara, and M.~Molina.
\newblock {\em Leavitt Path Algebras}.
\newblock Springer London, 2017.

\bibitem{Ara2024InverseSO}
P.~Ara, A.~Buss, and A.~D. Costa.
\newblock Inverse semigroups of separated graphs and associated algebras.
\newblock 2024.

\bibitem{Ash1975}
C.~J. Ash and T.~E. Hall.
\newblock Inverse semigroups on graphs.
\newblock {\em Semigroup Forum}, 11(1):140–145, Dec. 1975.

\bibitem{Boava17Inverse}
G.~Boava, G.~de~Castro, and L.~Mortari.
\newblock Inverse semigroups associated with labelled spaces and their tight
  spectra.
\newblock {\em Semigroup Forum}, 94, 06 2017.

\bibitem{Boava2021LeavittPA}
G.~Boava, G.~G. de~Castro, D.~Gonçalves, and D.~W. van Wyk.
\newblock Leavitt path algebras of labelled graphs.
\newblock {\em Journal of Algebra}, 2021.

\bibitem{CARLSEN20162090}
T.~M. Carlsen and N.~S. Larsen.
\newblock Partial actions and kms states on relative graph c⁎-algebras.
\newblock {\em Journal of Functional Analysis}, 271(8):2090--2132, 2016.

\bibitem{CORDEIRO2020917}
L.~G. Cordeiro and V.~Beuter.
\newblock The dynamics of partial inverse semigroup actions.
\newblock {\em Journal of Pure and Applied Algebra}, 224(3):917--957, 2020.

\bibitem{Castro2022CalgebrasOG}
G.~G. de~Castro and E.~J. Kang.
\newblock C*-algebras of generalized boolean dynamical systems as partial
  crossed products.
\newblock 2022.

\bibitem{DECASTRO2023126662}
G.~G. {de Castro} and E.~J. Kang.
\newblock Boundary path groupoids of generalized boolean dynamical systems and
  their c⁎-algebras.
\newblock {\em Journal of Mathematical Analysis and Applications},
  518(1):126662, 2023.

\bibitem{deCastro2020}
G.~G. de~Castro and D.~W. van Wyk.
\newblock Labelled space c*-algebras as partial crossed products and a
  simplicity characterization.
\newblock {\em Journal of Mathematical Analysis and Applications},
  491(1):124290, Nov. 2020.

\bibitem{Exel2007InverseSA}
R.~Exel.
\newblock Inverse semigroups and combinatorial c*-algebras.
\newblock {\em Bulletin of the Brazilian Mathematical Society, New Series},
  39:191--313, 2007.

\bibitem{Exel2007TightRO}
R.~Exel.
\newblock Tight representations of semilattices and inverse semigroups.
\newblock {\em Semigroup Forum}, 79:159--182, 2007.

\bibitem{Exel2003}
R.~Exel and M.~Laca.
\newblock Partial dynamical systems and the kms condition.
\newblock {\em Communications in Mathematical Physics}, 232(2):223–277, Jan.
  2003.

\bibitem{bd76e05f99e74532899de92ac94a9f73}
D.~Jones and M.~Lawson.
\newblock Graph inverse semigroups: their characterization and completion.
\newblock {\em Journal of Algebra}, 409:444--473, July 2014.

\bibitem{Kellendonk2004PartialAO}
J.~Kellendonk and M.~V. Lawson.
\newblock Partial actions of groups.
\newblock {\em Int. J. Algebra Comput.}, 14:87--114, 2004.

\bibitem{KELLENDONK2004462}
J.~Kellendonk and M.~V. Lawson.
\newblock Universal groups for point-sets and tilings.
\newblock {\em Journal of Algebra}, 276(2):462--492, 2004.

\bibitem{Lawson1998-ez}
M.~V. Lawson.
\newblock {\em Inverse Semigroups, the Theory of Partial Symmetries}.
\newblock World Scientific Publishing, Singapore, Singapore, Nov. 1998.

\bibitem{Lawson2009ANG}
M.~V. Lawson.
\newblock A noncommutative generalization of stone duality.
\newblock {\em Journal of the Australian Mathematical Society}, 88:385 -- 404,
  2009.

\bibitem{Lawson2011NonCommutativeSD}
M.~V. Lawson.
\newblock Non-commutative stone duality: Inverse semigroups, topological
  groupoids and c*-algebras.
\newblock {\em Int. J. Algebra Comput.}, 22, 2011.

\bibitem{lawson2022non}
M.~V. Lawson.
\newblock Non-commutative stone duality.
\newblock In {\em International Conference on Semigroups, Algebras, and
  Operator Theory}, pages 11--66. Springer, 2022.

\bibitem{MEAKIN2021107729}
J.~Meakin, D.~Milan, and Z.~Wang.
\newblock On a class of inverse semigroups related to leavitt path algebras.
\newblock {\em Advances in Mathematics}, 384:107729, 2021.

\bibitem{Meakin21GraphInverse}
J.~Meakin and Z.~Wang.
\newblock On graph inverse semigroups.
\newblock {\em Semigroup Forum}, 102, 02 2021.

\bibitem{Milan2008Semigroup}
D.~Milan.
\newblock {\em C*-algebras of inverse semigroups}.
\newblock PhD thesis, 2008.

\bibitem{Milan2011ONIS}
D.~Milan and B.~Steinberg.
\newblock On inverse semigroup c*-algebras and crossed products.
\newblock {\em Groups, Geometry, and Dynamics}, 8:485--512, 2011.

\bibitem{Paterson1999}
A.~L.~T. Paterson.
\newblock {\em Groupoids, Inverse Semigroups, and their Operator Algebras}.
\newblock Birkh\"{a}user Boston, 1999.

\bibitem{STEINBERG2010689}
B.~Steinberg.
\newblock A groupoid approach to discrete inverse semigroup algebras.
\newblock {\em Advances in Mathematics}, 223(2):689--727, 2010.

\end{thebibliography}
	
\end{document}